\documentclass[11pt]{amsart}

\usepackage{amsmath}
\usepackage{amsfonts}
\usepackage{amssymb}
\usepackage{amscd}
\usepackage{amsthm}
\usepackage{framed}
\usepackage{fullpage}
\usepackage{graphicx}
\usepackage{latexsym}
\usepackage[numbers]{natbib}  
\usepackage{longtable}
\usepackage{multirow}

\usepackage{hyperref}
\hypersetup{colorlinks,linkcolor={red},citecolor={blue},urlcolor={blue}}

\newtheorem{theorem}{Theorem}[section]
\newtheorem{lemma}[theorem]{Lemma}
\newtheorem{proposition}[theorem]{Proposition}

\newtheorem{conjecture}[theorem]{Conjecture}

\theoremstyle{definition}
\newtheorem{definition}[theorem]{Definition}
\newtheorem{example}[theorem]{Example}
\newtheorem{remark}[theorem]{Remark}

\numberwithin{equation}{section}

\newcommand{\CC}{\mathbb C}

\newcommand{\cD}{\mathcal D}
\newcommand{\cA}{\mathcal A}
\newcommand{\cH}{\mathcal H}
\newcommand{\cB}{\mathcal B}
\newcommand{\cE}{\mathcal E}
\newcommand{\PP}{\mathbb P}
\newcommand{\QQ}{\mathbb Q}
\newcommand{\RR}{\mathbb R}
\newcommand{\ZZ}{\mathbb Z}
\newcommand{\F}{\mathbb F}

\newcommand{\SL}{\mathop{\mathrm {SL}}\nolimits}
\newcommand{\SO}{\mathop{\mathrm {SO}}\nolimits}

\newcommand{\Orth}{\mathop{\null\mathrm {O}}\nolimits}
\newcommand{\U}{\mathop{\null\mathrm {U}}\nolimits}

\newcommand{\latt}[1]{{\langle{#1}\rangle}}
\newcommand{\ord}{\mathop{\mathrm {ord}}\nolimits}
\def\Grit{\operatorname{Grit}}

\def\div{\operatorname{div}}

\def\ker{\operatorname{ker}}
\def\det{\operatorname{det}}
\newcommand{\abs}[1]{\lvert#1\rvert}

\newenvironment{psmallmatrix}
  {\left(\begin{smallmatrix}}
  {\end{smallmatrix}\right)}

\begin{document}

\title{Free algebras of modular forms on ball quotients}

\author{Haowu Wang}

\address{Max-Planck-Institut f\"{u}r Mathematik, Vivatsgasse 7, 53111 Bonn, Germany}

\email{haowu.wangmath@gmail.com}
\author{Brandon Williams}

\address{Lehrstuhl A für Mathematik, RWTH Aachen, 52056 Aachen, Germany}

\email{brandon.williams@matha.rwth-aachen.de}


\date{\today}


\begin{abstract} 
In this paper we study algebras of modular forms on unitary groups of signature $(n,1)$. We give a necessary and sufficient condition for an algebra of unitary modular forms to be free in terms of the modular Jacobian. As a corollary we obtain a criterion that guarantees in many cases that, if $L$ is an even lattice with complex multiplication and the ring of modular forms for its orthogonal group is a polynomial algebra, then the ring of modular forms for its unitary group is also a polynomial algebra. We prove that a number of rings of unitary modular forms are freely generated by applying these criteria to Hermitian lattices over the rings of integers of $\QQ(\sqrt{d})$ for $d=-1,-2,-3$. As a byproduct, our modular groups provide many explicit examples of finite-covolume reflection groups acting on complex hyperbolic space. 
\end{abstract}

\maketitle

\begin{small}
\tableofcontents
\end{small}

\addtocontents{toc}{\setcounter{tocdepth}{1}} 

\section{Introduction}

In this paper we adapt our previous work on free algebras of modular forms on orthogonal groups to modular forms on ball quotients. It is expected that among modular varieties in general, that is, quotients of Hermitian symmetric domains $\cD$ by arithmetic subgroups $\Gamma$, there are only finitely many that are not of general type. (Indeed this is known to hold for orthogonal modular varieties by  theorems of Gritsenko--Hulek--Sankaran \cite{GHS07} and Ma \cite{Ma18}). In particular, one expects that there are only finitely many free algebras of modular forms, and determining them is an interesting open problem. We know from \cite{VP89} that the algebra of modular forms on $\cD$ for $\Gamma$ is free only if the group $\Gamma$ is generated by reflections. By \cite{Got69b, Mes72}, reflections exist only in two infinite families of symmetric domains: namely, complex balls and type IV symmetric domains in Cartan’s classification. Inspired by Vinberg's insight \cite{Vin13}, in \cite{Wan20} the first named author proved a sufficient and necessary condition for the algebra of modular forms on type IV symmetric domains for orthogonal groups of signature $(n,2)$ to be free. This asserts that the algebra of modular forms is free if and only if the modular Jacobian of (potential) generators is a cusp form vanishing precisely along mirrors of reflections with multiplicity one. Using this condition, we obtained a complete classification of such free algebras under some mild conditions (see \cite{Wan20}) and constructed free algebras of orthogonal modular forms for a large class of lattices (see \cite{Wan20a, WW20a, WW20b, WW20c}), which also cover most of the (many) previously known examples in the literature.

Compared to the case of orthogonal groups, much less is known about the rings of modular forms on ball quotients. In the literature one finds only a few explicit examples, e.g.  \cite{Fre02, FS14, FS15, FS19, H86, RT78, TR81, Shi88}. These modular forms are of interest in algebraic geometry and the associated modular varieties have been investigated by several mathematicians.  Allcock, Carlson and Toledo \cite{ACT02} showed that the moduli space of marked cubic surfaces is biholomorphic to the quotient of a 4-dimensional complex ball by a unitary group of of signature $(4, 1)$ over the Eisenstein integers. The modular forms for related groups were studied in \cite{AF02, Fre02}.   Also, Looijenga \cite{LS07} and Allcock--Carlson--Toledo \cite{ACT11} independently described the moduli space of stable cubic threefolds as an open subset of the quotient of a 10-dimensional complex ball by a discrete group of automorphisms.

This paper is also motivated by the construction of arithmetic complex hyperbolic reflection groups. It is difficult to construct finite-covolume reflection groups acting on complex hyperbolic spaces (see \cite{DM86, All00, All00a}) and there is no known classification of these groups. As suggested by Vinberg (see \cite[\S 10]{Bel16}), whenever one can find a free algebra of unitary modular forms, the underlying modular group is automatically an arithmetic group generated by complex reflections. 

We now outline the main results of the paper. 
We define the Jacobian of modular forms on $\U(n,1)$ and prove some of its basic properties (see Theorem \ref{th:Jacobian}). We then establish a sufficient and necessary condition for the algebra of unitary modular forms to be free. More precisely, we prove the following. 
\begin{theorem}[see Theorem \ref{th:freeJacobian} and Theorem \ref{th:modularJacobian}]\label{mth1}
Let $L$ be an even integral Hermitian lattice of signature $(n,1)$, $n>1$ over an imaginary quadratic field $\F=\QQ(\sqrt{d})$. Let $\Gamma$ be a finite-index subgroup of the unitary group $\U(L)$ of $L$. The algebra of modular forms for $\Gamma$ is free if and only if the Jacobian of a set of (potential) generators is a cusp form whose zero divisor is the sum of all mirrors of reflections in $\Gamma$ with a prescribed multiplicity determined by the order of reflections. 
\end{theorem}

In 1998 Borcherds \cite{Bor98} found a remarkable lift that constructs orthogonal modular forms whose zeros lie only on rational quadratic divisors. There is a natural embedding of $\U(n,1)$ into $\Orth(2n,2)$, and by pulling back Borcherds products along this embedding one can construct unitary modular forms vanishing only on mirrors of reflections. This argument was first used in \cite{All00a, AF02} and a full theory was developed in \cite{Hof14}. This provides a direct way to construct the Jacobian of the (potential) generators. 

The Jacobian criterion for free algebras of modular forms on unitary groups is similar to that for orthogonal groups. By relating the two notions of the modular Jacobian (see Proposition \ref{prop:nonzeroJ}) we find the following criterion. We continue to assume that $L$ is a Hermitian lattice of signature $(n,1)$, $n > 1$.

\begin{theorem}[see Theorem \ref{th:twins}]\label{mth2}
Let $d=-1$ or $-3$. Suppose the algebra of modular forms for the orthogonal group of the associated $\ZZ$-lattice of $L$ is free and that for $n$ of its generators the restrictions to the unitary group are identically zero. Then the algebra of modular forms for the unitary group of $L$ is freely generated by the restrictions of the remaining $n+1$ generators. This also holds when we replace the full orthogonal (unitary) group with the discriminant kernel of $L$. 
\end{theorem}
By applying this theorem to some specific Hermitian lattices, we construct 23 free algebras of unitary modular forms in a universal way. We present one example here. There is a unique Hermitian lattice $\mathrm{II}_{5,1}^\cE$ of signature $(5,1)$ over the Eisenstein integers whose underlying $\ZZ$-lattice is the even unimodular lattice $\operatorname{II}_{10,2}$. Hashimoto and Ueda \cite{HU14} proved that the algebra of modular forms for the orthogonal group of $\operatorname{II}_{10,2}$ is freely generated by 11 forms of weights 4, 10, 12, 16, 18, 22, 24, 28, 30, 36, 42. Since the weight of a nonzero modular form on $\U(\mathrm{II}_{5,1}^\cE)$ must be a multiple of $6$, the orthogonal generators of weight 4, 10, 16, 22, 28 vanish identically on the unitary symmetric domain. The above criterion implies that the algebra of modular forms for $\U(\mathrm{II}_{5,1}^\cE)$ is freely generated by 6 forms of weight 12, 18, 24, 30, 36, 42. As a corollary, we see that the group $\U(\mathrm{II}_{5,1}^\cE)$ is generated by reflections, which gives a new proof of a theorem of Allcock \cite{All00a}.

Even when the conditions of Theorem \ref{mth2} do not apply, the sufficient criterion of Theorem \ref{mth1} makes it possible to compute free algebras of unitary modular forms. In this way we construct more than 20 additional free algebras by giving their generators explicitly in terms of Borcherds products and additive theta lifts. 

In total, we construct more than 40 free algebras of unitary modular forms on Hermitian lattices over the Eisenstein integers and Gaussian integers. Most of them seem to be new. For these free algebras, the modular groups are generated by reflections and have finite index in the unitary group of the lattice, so they provide explicit examples of finite-covolume reflection groups acting on complex hyperbolic space.

Free algebras of unitary modular forms on Hermitian lattices over $\QQ(\sqrt{d})$ seem to be very exceptional when $d\neq -1$ and $-3$. We will see that no arithmetic subgroups of a discriminant kernel in such a Hermitian lattice contain complex reflections, so these never lead to free algebras of modular forms. Surprisingly, the ring of modular forms for the full unitary group of the root lattice $U\oplus  U(2) \oplus D_4$ viewed as a Hermitian lattice over $\mathbb{Q}(\sqrt{-2})$ is freely generated in weights $2, 8, 10, 16$. Our proof of this is essentially computational and is described briefly in section \ref{sec:sqrt2}. Despite some searching we know of no other examples.

This paper is structured as follows. In \S \ref{sec:preliminaries} we review the basic theory of modular forms on ball quotients. In \S \ref{sec:Jacobian} we prove Theorem \ref{mth1}. \S \ref{sec:twins} is devoted to the proof of Theorem \ref{mth2}. We construct free algebras of unitary modular forms not covered by Theorem \ref{mth2} case by case in the long section \S \ref{sec:more}. In \S \ref{sec:sqrt2} we describe the free algebra associated to $\U(3, 1)$ over $\mathbb{Q}(\sqrt{-2})$. In \S \ref{sec:nonfree} we present several interesting non-free algebras of unitary modular forms whose structure follows from the factorization of a Jacobian. Finally, we formulate some open questions and conjectures.

\section{Modular forms on ball quotients}\label{sec:preliminaries}
We first review Hermitian lattices and modular forms on complex ball quotients.

\subsection{Hermitian lattices}

Let $\F=\QQ(\sqrt{d})$ be an imaginary quadratic field with ring of integers $\mathcal{O}_\F$,  where $d$ is a square-free negative integer.  We denote by $D_\F$ the discriminant of $\F$, which is $d$ if $d\equiv 1\,(\bmod\,4)$ and $4d$ otherwise.  The maximal order in $\F$ is $$\mathcal{O}_\F=\ZZ+\ZZ\cdot \zeta, \; \text{where} \; \zeta=(D_{\F}+\sqrt{D_\F})/2.$$ The inverse different $\cD_\F^{-1}$ is the dual lattice of $\mathcal{O}_\F$ with respect to its trace form, and it is a principal fractional ideal: $$\cD_\F^{-1} = \{x \in \F: \; \mathrm{Tr}_{\F / \QQ}(xy) \in \mathbb{Z} \; \text{for all} \; y \in \mathcal{O}_\F\} = \frac{1}{\sqrt{D_{\F}}} \mathcal{O}_\F.$$

In this paper we will understand a \emph{Hermitian lattice} $L$ to be a free $\mathcal{O}_{\F}$-module together with a nondegenerate sesquilinear form $$\langle -, - \rangle : L \otimes L \longrightarrow \F.$$ (By convention, $\langle -, - \rangle$ is linear in the first component and conjugate-linear in the second.) To a Hermitian lattice $L$ we associate a $\mathbb{Z}$-lattice called the \emph{trace form} $L_{\ZZ}$, whose underlying abelian group is $L$ and whose bilinear form is $$(-, -) := \mathrm{Tr}_{\F/\QQ} \langle -, - \rangle : L \otimes L \longrightarrow \mathbb{Q}.$$

The Hermitian lattice $L$ is called integral resp. even if $L_{\ZZ}$ is (with respect to its induced bilinear form). Equivalently, $L$ is integral if $\langle -, - \rangle$ takes its values in $\cD_\F^{-1}$ and it is even if $\langle x,x \rangle \in \mathbb{Z}$ for all $x \in L$. The dual lattice $L' = (L_{\mathbb{Z}})'$ can be characterized with respect to the Hermitian form as $$L' = \{x \in L \otimes_{\mathcal{O}_{\F}} \F: \; \langle x, y \rangle \in \cD_\F^{-1} \;\text{for all} \; y \in L\}.$$

Suppose $L$ is even integral. The quotient $L'/L$ is a finite $\mathcal{O}_{\F}$-module and the Hermitian form $h(x) := \langle x, x \rangle$ induces a well-defined Hermitian form $$h : L'/L \longrightarrow \QQ / \ZZ.$$ The pair $(L'/L, h)$ is the \emph{discriminant form} attached to $L$.

We warn that our definition of integral Hermitian lattice is different from that of \cite{All00, All00a}. 

\subsection{Modular forms on \texorpdfstring{$\U(n, 1)$}{}} \label{sec:Taylor}
We always assume that $n\geq 2$. Let $L$ be a Hermitian lattice of signature $(n, 1)$ and define $V := L \otimes_{\mathcal{O}_{\F}} \CC$. The Hermitian symmetric domain attached to the unitary group $\U(V)\cong \U(n,1)$  is the Grassmannian of negative-definite lines: $$\cD_{\U} = \{[z] \in \PP(V): \; \langle z, z \rangle < 0\},$$ on which $\U(V)$ acts by multiplication. For any line $\ell \in \cD_{\U}$, let $\pi_{\ell}$ and $\pi_{\ell^{\perp}}$ denote the unitary projections to $\ell$ and $\ell^{\perp}$, respectively. The definition $$\phi_{\ell}([z]) := \pi_{\ell^{\perp}} \circ \Big( \pi_{\ell} \Big|_{[z]} \Big)^{-1} : \ell \rightarrow [z] \rightarrow \ell^{\perp}$$ determines a biholomorphic map $$\phi_{\ell} : \cD_{\U} \longrightarrow \cB = \{z \in \mathrm{Hom}(\ell, \ell^{\perp}): \; \|z\| < 1\}$$ to the unit ball in $\mathrm{Hom}(\ell, \ell^{\perp}) \cong \mathbb{C}^n$ under which $\ell$ itself is mapped to the origin. Through this identification the unitary group acts on $\cB$ as the group of M\"obius transformations. In this way the quotients $\cD_{\U} / \Gamma$ by arithmetic subgroups $\Gamma \le \U(V)$ can be understood as \emph{ball quotients}. 

Let $\cA_{\U}$ be the principal $\CC^*$-bundle $$\cA_{\U} = \{z \in V: \; [z] \in \cD_{\U}\}.$$

The arithmetic subgroups we will consider are finite-index subgroups $$\Gamma \le \U(L) := \{\gamma \in \U(V):\; \gamma L = L\}.$$ Besides $\U(L)$ itself, the most important example of such a subgroup is the \emph{discriminant kernel} $$\widetilde{\U}(L) = \{\gamma \in \U(L): \; \gamma x - x \in L \; \text{for all} \; x \in L'\}.$$

\begin{definition}
Let $\Gamma$ be a finite-index subgroup of $\U(L)$ and $k$ be a non-negative integer.  A modular form with respect to $\Gamma$ of weight $k$ and character $\chi: \Gamma \to \CC^*$ is a holomorphic function $F: \cA_{\U} \to \CC$ satisfying
\begin{align*}
F(tz)&=t^{-k}F(z),  \quad \text{for all $t\in \CC^*$},\\
F(\gamma z)&=\chi(\gamma) F(z),  \quad \text{for all $\gamma\in \Gamma$.}
\end{align*}
\end{definition}

By Baily--Borel \cite{BB66}, $\cD_{\U} / \Gamma$ is a quasi-projective variety of dimension $n$ over $\CC$.   This variety is compactified by including zero-dimensional cusps,  i.e.  $\Gamma$-orbits of isotropic vectors in $L\otimes \QQ$.  This yields the Satake--Baily--Borel compactification, denoted $(\cD_{\U}/ \Gamma)^*$, which can also be described as $\operatorname{Proj}(M_*(\Gamma))$ where $M_*(\Gamma)$ is the graded algebra of modular forms with respect to $\Gamma$ of integral weight and trivial character.

We will often be concerned with the question of whether a particular modular form $F$ is nonzero. To this end it is useful to consider various expansion of $F$ about cusps. (Note that every Hermitian lattice $L$ of signature $(n, 1)$ with $n \ge 2$ has cusps by the local-global principle.) Fix a nonzero isotropic vector $\ell \in L \otimes \QQ$ and a vector $\ell' \in L \otimes \QQ$ with $\latt{\ell, \ell'} = 1$, and write $z \in \cA_{\U}$ in the form $$z = \tau \ell - i \ell' + \mathfrak{z} \in L \otimes \CC, \; \text{where} \; \mathfrak{z} \perp \ell, \; \mathfrak{z} \perp \ell'.$$ The condition $\latt{z,z} < 0$ is equivalent to constraining $\tau$ to an upper half-plane of the form $\mathrm{im}(\tau) > C(\mathfrak{z})$ for every fixed $\mathfrak{z} \in \CC^{n-1}$. The space $$\cH_{\U} = \{(\tau, \mathfrak{z}) \in \CC \times \CC^{n-1}: \; \mathrm{im}(\tau) > C(\mathfrak{z})\}$$ is the \emph{Siegel domain} model of $\cD_{\U}$ with respect to $\ell$ and $\ell'$. 

For any modular form $F$ on $\cH_{\U}$, write $$F_{\ell, \ell'}(\tau, \mathfrak{z}) := F(\tau \ell - i \ell' + \mathfrak{z}).$$ (We will usually omit the index $\ell, \ell'$ from the notation.) The modularity of $F$ becomes a functional equation of the form $$F(\gamma \cdot (\tau, \mathfrak{z})) = j(\gamma; \tau, \mathfrak{z})^k F(\tau, \mathfrak{z}),$$ where $j$ is the automorphy factor determined by the induced action of $\U(V)$ on $\cH_{\U}$. Using this automorphy factor, we can define modular forms of rational weight with a multiplier system.

The function $F$ is invariant under $\tau \mapsto \tau + N$ for some period $N$, so it admits the \emph{Fourier--Jacobi expansion}
$$F(\tau, \mathfrak{z}) = \sum_{n \in \ZZ} a_n(\mathfrak{z}) e^{2\pi i n \tau / N}.$$
(Note that $a_n(\mathfrak{z})$ are not Jacobi forms in the usual sense.) Holomorphy at the cusp $[\ell]$ is equivalent to the vanishing of $a_n(\mathfrak{z})$ for $n < 0$; and the function $a_0(\mathfrak{z})$ is the value of $F$ at the cusp and is therefore a constant. The function $F$ is called a \emph{cusp form} if it vanishes at all cusps.

The Fourier--Jacobi expansion can be (and has been, cf. \cite{CG13}) used effectively for computations. In this paper we will instead represent $F$ by its Taylor series about $\mathfrak{z} = 0$. There are a few advantages; for one, the approach applies to all lattices equally, while a computational treatment of the Fourier--Jacobi coefficients seems to require a significant setup for each lattice; and also, for most forms of interest (restrictions to the unitary group of Eisenstein series and Gritsenko lifts, cf. \S\ref{sec:embedding}) the coefficients can be computed quite efficiently. Suppose first that $\ell'$ can be chosen to lie in $L$; that is, $L$ splits a unimodular plane over $\mathcal{O}_{\F}$. Then $\U(L)$ contains maps of the form $\gamma : (\tau, \mathfrak{z}) \mapsto ((a\tau + b)(c\tau + d)^{-1}, \mathfrak{z} (c \tau + d)^{-1})$ with $\begin{psmallmatrix} a & b \\ c & d \end{psmallmatrix} \in \mathrm{SL}_2(\mathbb{Z})$, with factor of modularity $j(\gamma; \tau, \mathfrak{z}) = c \tau + d$. In other words, $$F \Big( \frac{a \tau + b}{c \tau + d}, \frac{\mathfrak{z}}{c \tau + d} \Big) = (c \tau + d)^k F(\tau, \mathfrak{z}).$$ This means that $F$ has a Taylor expansion $$F(\tau, \mathfrak{z}) = \sum_{\alpha} f_{\alpha}(\tau) \mathfrak{z}^{\alpha}, \; \; \|\mathfrak{z}\| < c(\tau),$$ converging in a neighborhood of each $\tau \in \mathbb{H}$, in which the coefficient $f_{\alpha}$ is a modular form for $\mathrm{SL}_2(\mathbb{Z})$ of weight $k + |\alpha|$. (Here $\mathfrak{z}^{\alpha} = z_1^{\alpha_1}...z_n^{\alpha_n}$ and $|\alpha|= \alpha_1 + ... + \alpha_n$ for $\mathfrak{z} = (z_1,...,z_n)$ and $\alpha = (\alpha_1,...,\alpha_n)$.) In the case of general $\ell'$, the Taylor expansion can be defined similarly but the coefficients will be modular only under a congruence subgroup.

\subsection{Reflections}\label{sec:reflections}
 Let $L$ be an even integral Hermitian lattice over $\mathcal{O}_\F$. Reflections are automorphisms of finite order whose fixed point set has codimension one.  For any primitive vector $r\in L$ with $\latt{r,r}> 0$ and $\alpha \in \mathcal{O}_\F^\times \setminus \{1\}$, the reflection associated to $r$ and  $\alpha$ is
$$
\sigma_{r, \alpha}: x \mapsto x -(1-\alpha) \frac{ \latt{x, r}}{\latt{r,r}} r .
$$
The hyperplane orthogonal to $r$ is called the mirror of $\sigma_{r, \alpha}$, denoted $r^\perp$. 
Note that the composition of two reflections along the same mirror is $$\sigma_{r, \alpha} \sigma_{r, \beta} = \ell \mapsto \ell - \Big((1 - \alpha) + (1 - \beta) - (1 - \alpha)(1 - \beta) \Big) \frac{\latt{\ell, r}}{\latt{r, r}} r = \sigma_{r, \alpha \beta},$$ or the identity when $\alpha \beta = 1$. In particular, $$\mathrm{ord}(\sigma_{r, \alpha}) = \mathrm{ord}(\alpha) \in \{2, 3, 4, 6\},$$ where $\mathrm{ord}(\alpha)$ is the order of $\alpha$ in $\mathcal{O}_\F^{\times}$. Reflections of order two, three, four, six are also called \emph{biflections}, \emph{triflections}, \emph{tetraflections} and \emph{hexaflections}, respectively.

We will first describe the reflections in the discriminant kernel.
\begin{lemma}\label{lem:kernelreflection}
Let $r$ be a primitive vector of positive norm in $L$ and $\alpha \in \mathcal{O}_\F^\times \setminus \{1\}$. 
\begin{enumerate}
\item When $d=-1$,  $\sigma_{r, \alpha} \in \widetilde{\U}(L)$ if and only if $\latt{r,r}=1$ and $\alpha=-1$;
\item When $d=-3$,  $\sigma_{r, \alpha} \in \widetilde{\U}(L)$ if and only if $\latt{r,r}=1$ and $\alpha \in \{e^{2\pi i / 3}, e^{4\pi i / 3}\}$;
\item When $d\neq -1, -3$,  there are no reflections in $\widetilde{\U}(L)$. 
\end{enumerate}
\end{lemma}

\begin{proof}
$\sigma_{r, \alpha} \in \widetilde{\U}(L)$ if and only if
$$
(1-\alpha)\frac{ \latt{\ell, r}}{\latt{r,r}} \in \mathcal{O}_\F   \; \text{for all $\ell\in L'$}. 
$$
We define 
$$
\latt{L',r}=\{ \latt{\ell, r}: \ell \in L' \}.
$$
Then $\sqrt{D_{\F}} \latt{L',  r} \subset \mathcal{O}_\F$ is an ideal which is not divisible by any proper principal ideal because $r$ is primitive.  However if $\sigma_{r, \alpha} \in \widetilde{\U}(L)$ then $\sqrt{D_{\F}} \latt{L',  r}$ is divisible by $\sqrt{D_{\F}} \latt{r,r} / (1-\alpha)$, so the latter must be a unit.  By comparing the possible norms of $\sqrt{D_{\F}}\latt{r,r}$ and $(1-\alpha)$,  we obtain (3) and the necessity of the conditions in (1) and (2).  That the conditions in (1) and (2) are sufficient is easy to verify directly.
\end{proof}

The above lemma implies that, when $d\neq -1$ and $-3$, the unitary group of any Hermitian lattice over $\mathcal{O}_\F$ whose associated $\ZZ$-lattice is an even unimodular lattice $\mathrm{II}_{8m+2,2}$  contains no reflections. 

Now we will describe the reflections in the full unitary group of a Hermitian lattice. To do this we have to discuss the relationship between the unitary reflections $\sigma_{r,\alpha}$ and their associated orthogonal reflection $$
\sigma_{r}: \ell \mapsto \ell - \frac{2(\ell,r)}{(r,r)}.
$$ The cases $\F=\QQ(\sqrt{-1})$ and $\F = \QQ(\sqrt{-3})$ are special and must be treated separately.
\begin{lemma}\label{lem:Gaussian}
Assume that $\F=\QQ(\sqrt{-1})$. Let $r$ be a primitive vector of positive norm in $L$.
\begin{enumerate}
    \item $\sigma_{r,-1}\in \U(L)$ if and only if $\sigma_r\in \Orth^+(L_{\ZZ})$;
    \item $\sigma_{r,i} \in \U(L)$ if and only if $\sigma_{(1+i)r}\in \Orth^+(L_{\ZZ})$.
\end{enumerate}
\end{lemma}
\begin{proof}
(1) Since $r\in L$ is primitive, $r$ and $(1+i)r$ are both primitive in $L_{\ZZ}$. By definition, $\sigma_{r,-1}\in \U(L)$ if and only if $2\latt{L,r}/\latt{r,r} \subseteq \ZZ[i]$, that is, $r/\latt{r,r} \in L'$. This is equivalent to $$(L,r/\latt{r,r})=2(L,r)/(r,r) \subseteq \ZZ,$$ or $\sigma_r\in \Orth^+(L_{\ZZ})$. 

(2) $\sigma_{r,i} \in \U(L)$ if and only if $$(1-i)\frac{\latt{L,r}}{\latt{r,r}} \subseteq \ZZ[i],$$ or equivalently, $(1+i)r/(2\latt{r,r}) \in L'$. On the other hand, $\sigma_{(1+i)r}\in \Orth^+(L_{\ZZ})$ if and only if $$2 \frac{(L,(1+i)r)}{((1+i)r,(1+i)r)}= \frac{(L,(1+i)r)}{(r,r)}\subseteq \ZZ,$$ which is equivalent to $(1+i)r/(r,r) = (1+i)r / (2 \langle r, r \rangle) \in L'$.
\end{proof}

Now we consider the case $\F=\QQ(\sqrt{-3})$.
\begin{lemma}\label{lem:Eisenstein}
Assume that $\F=\QQ(\sqrt{-3})$. Let $r$ be a primitive vector of positive norm in $L$.
\begin{enumerate}
    \item $\sigma_{r,-\omega}\in \U(L)$ if and only if $\sigma_r\in \Orth^+(L_{\ZZ}), \; \text{where} \; \omega = e^{\pi i / 3}$;
    \item $\sigma_{r,\omega} \in \U(L)$ if and only if $\sigma_{(1+\omega)r}\in \Orth^+(L_{\ZZ})$;
    \item $\sigma_{r,-1}\in \U(L)$ and $\sigma_{r,-\omega}\not\in\U(L)$ if and only if there exists a positive integer $t$ such that $(r,r)=4t$ and $(L,(1+\omega)r)=3t\ZZ$.
\end{enumerate}
\end{lemma}
\begin{proof}
(1) The assumption that $r\in L$ is primitive implies that $r$ and $(1+\omega)r$ are both primitive in $L_{\ZZ}$. By definition, $\sigma_{r,-\omega}\in \U(L)$ if and only if $(1+\omega)\latt{L,r}/\latt{r,r}\subseteq \ZZ[\omega]$, or equivalently
    $$
    \frac{(1+\overline{\omega})r}{\sqrt{-3}\latt{r,r}}=-\omega \cdot \frac{r}{\latt{r,r}}\in L'.
    $$
    This is equivalent to $r/\latt{r,r}\in L'$ and therefore $\sigma_r\in \Orth^+(L_{\ZZ})$.

(2) $\sigma_{r,\omega}\in \U(L)$ if and only if $(1-\omega)\latt{L,r}/\latt{r,r}\in \ZZ[\omega]$. This is equivalent to $\latt{L,r}/\latt{r,r}\subseteq \ZZ[\omega]$, and therefore to
    $$
    \frac{r}{\sqrt{-3}\latt{r,r}}=-\omega \cdot \frac{(1+\omega)r}{3\latt{r,r}}\in L'.
    $$
    On the other hand, $\sigma_{(1+\omega)r}\in \Orth^+(L_{\ZZ})$ if and only if $$2\frac{(L,(1+\omega)r)}{((1+\omega)r,(1+\omega)r)} \subseteq \ZZ,$$ that is, $$\frac{(1+\omega)r}{3\latt{r,r}}\in L'.$$ This implies the claim.
    
(3)    Write $(r,r)=2a$ and $(L,(1+\omega)r)=b\ZZ$ for some positive integers $a$ and $b$. Since $(r,(1+\omega)r)=3a$, it follows that $b|3a$.
    Now $\sigma_{r,-1}\in \U(L)$ if and only if $2\latt{L,r}/\latt{r,r}\subseteq \ZZ[\omega]$; that is,
    $$
    \frac{2r}{\sqrt{-3}\latt{r,r}}=-\omega\cdot\frac{2(1+\omega)r}{3\latt{r,r}} \in L'.
    $$
    This is equivalent to 
    $$
    \left(L, \frac{2(1+\omega)r}{3\latt{r,r}} \right) \subseteq \ZZ, \; \text{i.e.} \; \frac{4(L,(1+\omega)r)}{3(r,r)} \subseteq \ZZ.
    $$
    This is equivalent to $3a=2b$ or $3a=b$. The case $3a = b$ occurs if and only if $$\sigma_{(1 + \omega)r} \in \Orth^+(L_{\ZZ}),$$ or in other words (by part (2)) $\sigma_{r, \omega} \in \U(L)$. Therefore $\sigma_{r, -1} \in \U(L)$ and $\sigma_{r, -\omega} \notin \U(L)$ if and only if there exists a positive integer $t$ such that $a=2t$ and $b=3t$, i.e. \[ (r, r) = 4t \; \text{and} \; (L, (1 + \omega)r) = 3t \mathbb{Z}. \qedhere \]
\end{proof}

Observe in particular that biflections of type (3) do \emph{not} have corresponding orthogonal reflections: if $\sigma_{r, -1} \in \U(L)$ and $\sigma_{r, -\omega} \notin \U(L)$, then  $$\frac{2(L, (1 + \omega)r)}{3 (r, r)} = \frac{1}{6t} \cdot 3t \mathbb{Z} \not \subseteq \mathbb{Z}$$ and therefore $\sigma_{(1 + \omega)r} \not\in \Orth^+(L_{\ZZ})$; also, by part (1), $\sigma_{r, -\omega} \notin \U(L)$ implies $\sigma_r \not\in \Orth^+(L_{\ZZ})$. Biflections of this sort do exist and several examples will appear throughout the paper.

Finally, we will describe the reflections over other imaginary quadratic fields.

\begin{lemma}\label{lem:other23}
Assume that $d\equiv 2,3 \mod 4$ and $d\neq -1$. Let $r$ be a primitive vector of positive norm in $L$. Then $\sigma_{r,-1}\in \U(L)$ if and only if $\sigma_{\sqrt{d}r}$ belongs to $\Orth^+(L_{\ZZ})$. In this case $\sigma_r$ also belongs to $\Orth^+(L_{\ZZ})$.
\end{lemma}
\begin{proof}
$\sigma_{r,-1}\in \U(L)$ if and only if $2\latt{L,r}/\latt{r,r}\subseteq \ZZ[\sqrt{d}]$; that is, $r/(\sqrt{d}\latt{r,r})\in L'$. On the other hand, $\sigma_r\in \Orth^+(L_{\ZZ})$ if and only if $r/\latt{r,r}\in L'$, and $\sigma_{\sqrt{d}r}\in \Orth^+(L_{\ZZ})$ if and only if $r/(\sqrt{d}\latt{r,r})\in L'$. This proves the lemma.
\end{proof}

\begin{lemma}\label{lem:other1}
Assume that $d\equiv 1 \mod 4$ and $d\neq -3$. Let $r$ be a primitive vector of positive norm in $L$. 
\begin{enumerate}
    \item $\sigma_{r,-1}\in \U(L)$ and $\sigma_r \in \Orth^+(L_{\ZZ})$ if and only if $\sigma_{\sqrt{d}r}\in \U(L)$;
    \item $\sigma_{r,-1}\in \U(L)$ and $\sigma_r \not\in \Orth^+(L_{\ZZ})$ if and only if there exists a positive integer $t$ such that $(r,r)=4t$ and $(L,\sqrt{d}r)=dt\ZZ$. 
\end{enumerate}
\end{lemma}
\begin{proof}
(1) $\sigma_{r,-1}\in \U(L)$ if and only if $2r/(\sqrt{d}\latt{r,r})\in L'$, and $\sigma_r\in \Orth^+(L_{\ZZ})$ if and only if $r/\latt{r,r}\in L'$. On the other hand, $\sigma_{\sqrt{d}r}\in \Orth^+(L_{\ZZ})$ if and only if $r/(\sqrt{d}\latt{r,r})\in L'$. Since 
$$
\frac{1+\sqrt{d}}{2}\cdot \frac{2r}{\sqrt{d}\latt{r,r}} - \frac{r}{\latt{r,r}} = \frac{r}{\sqrt{d}\latt{r,r}},
$$
the claim follows.

(2) Suppose $(r,r)=2a$ and $(L, \sqrt{d}r)=b\ZZ$ for some positive integers $a$ and $b$. Since $$((1+\sqrt{d})r/2,\sqrt{d}r)=-da,$$ it follows that $b|ad$. We see that $\sigma_{r,-1}\in \U(L)$ if and only if 
$$
\frac{4(L,\sqrt{d}r)}{d(r,r)} \subseteq \ZZ, \; \text{or equivalently} \; ad=2b \; \text{or} \; ad=b.
$$
The case $ad=b$, (i.e. the condition in (2) does not hold) occurs if and only if $$2(L,\sqrt{d}r)/(d(r,r))=\ZZ,$$ or equivalently $\sigma_{\sqrt{d}r}\in \Orth^+(L_{\ZZ})$. By (1) this is equivalent to $\sigma_r\in \Orth^+(L_{\ZZ})$.
\end{proof}

\subsection{Embedding U(n,1) in O(2n,2)}\label{sec:embedding}

Suppose $L$ is a Hermitian lattice of signature $(n, 1)$. There is a natural embedding of $\U(L)$ in the orthogonal group $\Orth(L_{\mathbb{Z}})$ which allows us to construct unitary modular forms as restrictions of orthogonal modular forms. We will review this in more detail.

Conceptually it is more convenient to take the underlying $\ZZ$-lattice $L_{\ZZ}$ and to express the action of $\mathcal{O}_{\F}$ on $L$ as a complex structure on $V := L_{\ZZ} \otimes_{\ZZ} \RR$: $$J \in \mathrm{SO}(V) \; \text{with} \; J^2 = -\mathrm{Id}.$$ In particular $L_{\ZZ}$ admits the endomorphism $$\zeta := \frac{D_\F}{2} + \frac{\sqrt{|D_\F|}}{2} J.$$ The Hermitian form on $L$ can be recovered from the bilinear form $(-, -)$ on $L_{\ZZ}$ via $$\latt{x, y} = \frac{1}{2}(x, y)  + \frac{1}{2i} (Jx, y).$$ The unitary group is then $$\U(V) = \{\gamma \in \mathrm{O}(V): \; J \gamma J = -\gamma\} \subseteq \Orth(V).$$
Since $\U(V)$ is connected, one has $\U(V)\subseteq \SO(V)$.

The Hermitian symmetric domain attached to $\Orth(V)$ is a connected component $$\cD^+(L) \subseteq \cD(L) := \{[z] \in \PP(L \otimes_{\ZZ} \mathbb{C}): \; (z, z) = 0, \; (z, \overline{z}) < 0\}.$$ Points $[z] = [x+iy] \in \PP(L \otimes \CC)$ may be identified with oriented negative-definite planes $\mathrm{span}(x, y) \subseteq V$. The type I domain $\cD_{\U}(L)$ is the subset of those planes which are complex lines (with respect to the complex structure $J$): $$\cD_{\U}(L) = \{[x - i Jx]: \; x \in V, \; (x,x) < 0\} \subseteq \cD^+(L).$$

Orthogonal modular forms are homogeneous holomorphic functions on the affine cone $$\cA(L) = \{z \in L \otimes_\ZZ \CC: \; [z] \in \cD^+(L)\}.$$ Any orthogonal modular form $F : \cA(L) \rightarrow \mathbb{C}$ for an arithmetic subgroup $\Gamma \le \Orth^+(L)$ restricts to a unitary modular form $$F_{\U} : \cA_{\U}(L) = \{z \in \cA(L): \; [z] \in \cD_{\U}(L)\} \longrightarrow \CC$$ for the group $\Gamma_{\U} := \Gamma \cap \U(L)$ of the same weight and character.

In this paper we will only need to consider lattices whose trace form decomposes as $$L_{\ZZ} = (\mathcal{O}_\F v_1 + \mathcal{O}_\F v_2) \oplus L_0,$$ where $\mathcal{O}_\F v_1, \mathcal{O}_\F v_2$ are isotropic planes that admit multiplication by $\mathcal{O}_\F$ and $L_0$ is positive-definite. In this case we can identify orthogonal modular forms with their restrictions to the tube domain of vectors in $\cA(L)$ of the form $$(\tau, \mathfrak{z}, w) := \tau v_1 + w v_2 + \lambda \zeta(v_1) + \zeta(v_2) + \mathfrak{z},$$ where $\tau, w \in\mathbb{H}$ and $\mathfrak{z} \in L_0 \otimes \mathbb{C}$, and where $\lambda \in \mathbb{C}$ is uniquely determined by the fact that this vector has norm zero. A vector of this form lies in $\mathcal{A}_{\U}(L)$ if and only if $J \mathfrak{z} = i \mathfrak{z}$ and if $$w v_2 + \zeta(v_2) = (w + D_{\F}/2) v_2 + (\sqrt{|D_{\F}|}/2) Jv_2 = x - i Jx \; \text{for some}\; x \in L \otimes \mathbb{C},$$ i.e. $$w = \frac{-D_{\F} + \sqrt{D_{\F}}}{2} =: -\overline{\zeta}.$$ It is natural to decompose $\mathfrak{z}$ further into its holomorphic and antiholomorphic parts: $$\mathfrak{z} = \mathfrak{z}_{\text{hol}} + \mathfrak{z}_{\text{anti}}, \; \; \mathfrak{z}_{\text{hol}} = \frac{\mathfrak{z} - i J \mathfrak{z}}{2}, \; \mathfrak{z}_{\text{anti}} = \frac{\mathfrak{z} + i J \mathfrak{z}}{2}.$$ If we write an orthogonal modular form $F$ as a function on the tube domain in the form $$F(\tau, \mathfrak{z}_{\text{hol}}, \mathfrak{z}_{\text{anti}}, w),$$ then its restriction to the unitary group is $$F_{\U}(\tau, \mathfrak{z}_{\text{hol}}) := F(\tau, \mathfrak{z}_{\text{hol}}, 0, -\overline{\zeta}), \; \text{where}\; \tau \in \mathbb{H}, \; |\mathfrak{z}_{\text{hol}}| \; \text{sufficiently small}.$$

\subsection{Borcherds products}

Let $(L,h)$ be an even integral Hermitian lattice of signature $(n, 1)$ over $\mathcal{O}_{\F}$. For any vector $r \in L'$ of positive norm, the \emph{rational quadratic divisors} $r^{\perp}$ on $\mathcal{D}^+(L)$ and $\cD_{\U}(L)$ are the intersections $$r^{\perp} \cap \cD^+(L) = \{[z] \in \cD^+(L): \; (z, r) = 0\}$$ and $$r^{\perp} \cap \cD_{\U}(L) = \{[z] \in \cD_{\U}(L): \; \latt{z, r} = 0\} = (r^{\perp} \cap \cD^+(L)) \cap \cD_{\U}(L).$$ For any coset $\beta \in L'/L$ and any positive $m \in \mathbb{Z} + h(\beta)$, the \emph{Heegner divisors} are the $\widetilde{\Orth}^+(L)$-invariant resp. $\widetilde{\U}(L)$-invariant divisors $$\mathbf{H}(m, \beta) = \bigcup_{\substack{r \in L + \beta \\ h(r) = m}} r^{\perp} \cap \cD^+(L),$$ $$\mathbf{H}_{\U}(m, \beta) = \bigcup_{\substack{r \in L + \beta \\ h(r) = m}} r^{\perp} \cap \cD_{\U}(L).$$

In 1998 Borcherds \cite{Bor98} established a remarkable multiplicative lift which sends nearly holomorphic modular forms for the Weil representation of $\SL_2(\ZZ)$ to meromorphic orthogonal modular forms (called \emph{Borcherds products}) with infinite product expansions and explicit divisors that are supported on rational quadratic divisors.  Using the embedding of $\U(n,1)$ in $\Orth(2n,2)$, in \cite{Hof14} Hofmann constructed Borcherds products on unitary groups. This lift is essential for our arguments.

Note that the restriction of a Borcherds product on $\Orth(2n, 2)$ to $\U(n, 1)$ never vanishes identically. However, it is quite possible for a nontrivial Borcherds product of weight $0$ on $\Orth(2n, 2)$ to descend to a constant function on $\U(n, 1)$. This is because the unitary Heegner divisors satisfy inclusions $$\mathbf{H}_{\U}(m, \beta) \subseteq \mathbf{H}_{\U}(|\lambda|^2 m, \lambda \beta), \; \; \lambda \in \mathcal{O}_{\F} \backslash \{0\},$$ while the orthogonal Heegner divisors satsify these inclusions only for $\lambda \in \mathbb{Z}$.

For lattices over the Eisenstein or Gaussian integers an additional complication may arise. For any nontrivial unit $\alpha \in \mathcal{O}_{\F}^{\times} \backslash \{\pm 1\}$, the hyperplanes $r^{\perp}$ and $(\alpha r)^{\perp}$ define the same divisor in $\cD_{\U}(L)$ but distinct divisors in $\cD^+(L)$. In particular a Borcherds product which is irreducible as an orthogonal modular form may admit a holomorphic square (when $\mathcal{O}_{\F} = \mathbb{Z}[i]$) or cube (when $\mathcal{O}_{\F} = \mathbb{Z}[e^{2\pi i / 3}]$) root as a unitary modular form.

\begin{example}
Consider the signature $(2,1)$ Hermitian lattice $$L=H(2)\oplus \mathcal{O}_{\F}v$$ over $\F=\QQ(\sqrt{-1})$, where $v$ satisfies $\latt{v,v}=1$ and where $H(2)=\mathcal{O}_{\F}\ell+\mathcal{O}_{\F}\ell'$ is the $2$-rescaling of a hyperbolic plane over $\mathcal{O}_{\F}=\ZZ[i]$ satisfying $$\latt{\ell,\ell}=\latt{\ell,\ell'}=0,\; \latt{\ell,\ell'}=i, \; \latt{\ell',\ell}=-i.$$ The corresponding $\ZZ$-lattice $L_{\ZZ}$ is $2U(2)\oplus 2A_1$, with Gram matrix $\begin{psmallmatrix} 0 & 0 & 0 & 0 & 0 & 2 \\ 0 & 0 & 0 & 0 & 2 & 0 \\ 0 & 0 & 2 & 0 & 0 & 0 \\ 0 & 0 & 0 & 2 & 0 & 0 \\ 0 & 2 & 0 & 0 & 0 & 0 \\ 2 & 0 & 0 & 0 & 0 & 0 \end{psmallmatrix}$ with respect to the basis $\{\ell, -i\ell, v, iv, \ell', i\ell'\}$. Consider the Borcherds lifts on $\widetilde{\Orth}^+(L_{\ZZ})$ of the vector-valued modular forms with the following principal parts at $\infty$:
\begin{align*}
&F_1:& &2e_0+q^{-1/4}(e_{(0,0,1/2,0,0,0)}+e_{(0,0,0,1/2,0,0)});&\\
&F_2:& &2e_0+q^{-1/4}(e_{(0,1/2,1/2,0,0,0)}+e_{(0,1/2,0,1/2,0,0)});&\\
&F_3:& &2e_0+q^{-1/4}(e_{(1/2,0,1/2,0,0,0)}+e_{(1/2,0,0,1/2,0,0)});&\\
&F_4:& &8e_0+q^{-1/2}e_{(0,0,1/2,1/2,0,0)}.&
\end{align*}
Their divisors are
\begin{align*}
\div(F_1) &=\mathbf{H}(1/4, v/2) + \mathbf{H}(1/4, iv/2);  \\
\div(F_2) &=\mathbf{H}(1/4, i\ell/2 + v/2) + \mathbf{H}(1/4, i\ell/2 + iv/2);  \\
\div(F_3) &=\mathbf{H}(1/4, \ell/2 + v/2) + \mathbf{H}(1/4, \ell/2 + iv/2);  \\
\div(F_4) &=\mathbf{H}(1/2, (1+i)v/2).
\end{align*}
Observe that
\begin{align*}
&i\cdot \mathbf{H}(1/4, v/2) = \mathbf{H}(1/4, iv/2),& &i\cdot \mathbf{H}(1/2, (1+i)v/2) = \mathbf{H}(1/2, (1+i)v/2),& \\
&i\cdot \mathbf{H}(1/4, i\ell/2 + v/2)=\mathbf{H}(1/4, \ell/2 + iv/2),& &i\cdot \mathbf{H}(1/4, i\ell/2 + iv/2) = \mathbf{H}(1/4, \ell/2 + v/2).&
\end{align*}
Let $f_i$ be the restriction of $F_i$ to $\cD_{\U}(L)$. Then $f_1$ has multiplicity two zeros exactly along the hyperplanes in $\mathbf{H}_{\U}(1/4, v/2)$ so its square root is a well-defined modular form of weight $1/2$. The functions $f_2$ and $f_3$ have simple zeros along the same hyperplanes so by Koecher's principle $f_2$ equals $f_3$ up to scalar. Finally, $f_4$ has only double zeros; moreover, if $r^\perp$ lies in $\div(f_1)$, then $r^{\perp} = ((1+i)r)^\perp$ also lies in $\div(f_4)$. Therefore $f_4/f_1$ is holomorphic even though $F_4/F_1$ is not.
\end{example}

\section{Free algebras of modular forms and the Jacobian}\label{sec:Jacobian}
In \cite{AI05} Aoki and Ibukiyama defined the Jacobian of Siegel modular forms. In \cite{Wan20} the first named author proved some properties of the Jacobian of orthogonal modular forms and established a necessary and sufficient condition for an algebra of orthogonal modular forms being free. This condition is essentially the existence of the Jacobian of (potential) generators vanishing exactly on all mirrors of reflections with multiplicity one. In this section we define the Jacobian of unitary modular forms and prove an analogous criterion for algebras of modular forms on unitary groups.

Let  $\Gamma$ be a finite-index subgroup of $\U(L)$.   For $1 \leq j \leq n+1$,  let  $F_j$ be a modular form with respect to $\Gamma$ of weight $k_j$ and trivial character.  We view $F_j$ as holomorphic functions defined on the affine cone $\cA_{\U} \subseteq \CC^{n, 1}$.  With respect to the coordinates $(z_1,...,z_{n+1})$ on $\CC^{n, 1}$, the Jacobian of $(F_1,...,F_{n+1})$ is defined in the usual way:
$$
J_{\U}:=J_{\U}(F_1, ...,F_{n+1})=\frac{\partial (F_1,  F_2, ..., F_{n+1})}{\partial (z_1, z_2, ..., z_{n+1})}. 
$$

\begin{theorem}\label{th:Jacobian}
\begin{enumerate}
\item The Jacobian $J_{\U}$ is a cusp form for $\Gamma$ of weight $n+1 + \sum_{j=1}^{n+1}k_j$ and the character $\det^{-1}$, where $\det$ is the determinant character. 
\item The Jacobian $J_{\U}$ is not identically zero if and only if these forms $F_j$ are algebraically independent over $\CC$.
\item Let $r\in L$ and $\alpha \in \mathcal{O}_\F^\times \setminus \{1\}$.  If $\sigma_{r,\alpha}\in \Gamma$,  then $$\mathrm{ord}(J_{\U}, r^\perp) \equiv -1 \, \mathrm{mod}\, \ord(\alpha),$$ where $\ord(\alpha)$ is the order of $\alpha$ in the group of units.
\item If we view $F_j$ as functions of $\tau, z_1,...,z_{n-1}$ on the Siegel domain $\cH_{\U}$ attached to a zero-dimensional cusp, then the Jacobian takes the form
$$
\left\lvert \begin{array}{cccc}
k_1F_1 & k_2F_2 & \cdots & k_{n+1}F_{n+1} \\ 
\frac{\partial F_1}{\partial \tau} & \frac{\partial F_2}{\partial \tau} & \cdots & \frac{\partial F_{n+1}}{\partial \tau} \\
\frac{\partial F_1}{\partial z_1} & \frac{\partial F_2}{\partial z_1} & \cdots & \frac{\partial F_{n+1}}{\partial z_1} \\ 
\vdots & \vdots & \ddots & \vdots \\ 
\frac{\partial F_1}{\partial z_{n-1}} & \frac{\partial F_2}{\partial z_{n-1}} & \cdots & \frac{\partial F_{n+1}}{\partial z_{n-1}}
\end{array}   \right\rvert.
$$
\end{enumerate}
\end{theorem}

\begin{proof}
(1) Clearly $J_{\U}$ is a holomorphic function on $\cA_{\U}$.  The chain rule shows that it is homogeneous of degree $-(n+1) - \sum_{j=1}^{n+1} k_j$ and transforms by $$J_{\U}(\gamma z) = \det^{-1}(\gamma) J_{\U}(z), \; \gamma \in \Gamma.$$ Therefore it is a modular form for $\Gamma$ of the indicated weight and character. At any zero-dimensional cusp, using the expression for the Jacobian on the Siegel domain $\cH_{\U}$ (part (4) in this Theorem) and the Fourier--Jacobi expansions of $F_1,...,F_{n+1}$ shows that $J_{\U}$ vanishes; therefore $J_{\U}$ is a cusp form.

(2)  This is the same argument as \cite[Theorem 2.5 (2)]{Wan20}. 

(3) Let $F\in M_k(\Gamma,\chi)$ be a unitary modular form for some character $\chi$, and suppose $F$ vanishes along $r^{\perp}$ to order $m \in \mathbb{N}_0$. Write vectors $z \in \cA_{\U}$ in the form $$z = z_1 \cdot r + z_2, \; \text{where} \; z_1 \in \CC, \; z_2 \in r^\perp,$$ so the Taylor expansion of $F$ about $r^{\perp}$ takes the form $$F(z) = F(z_1 r + z_2) = z_1^m F(z_2) + O(z_1^{m+1}).$$ Then
\begin{align*}
F(\sigma_{r,\alpha}(z) ) & =  F(\alpha z_1\cdot r + z_2) = (\alpha z_1)^m f(z_2) + O(z_1^{m+1}).
\end{align*}
When $\sigma_{r, \alpha} \in \Gamma$ and $\chi = \mathrm{det}^{-1}$, comparing this to the Taylor expansion
$$F(\sigma_{r, \alpha}(z)) = \mathrm{det}^{-1}(\sigma_{r, \alpha}) F(z) = \alpha^{-1} \Big( z_1^m f(z_2) + O(z_1^{m+1}) \Big)$$ yields $\alpha^m = \alpha^{-1}$ and therefore $\mathrm{ord}(\alpha) | (m+1).$

(4) Choose a zero-dimensional cusp with associated Siegel domain $\cH_{\U}$. Let $f_j$ be unitary modular forms defined on the domain $\cH_{\U}$. With respect to a suitable coordinate system the functions $$F_j(\tau, z_1,...,z_n):=z_n^{-k_j}f_j(\tau,z_1,...,z_{n-1})$$ are modular forms on the affine cone $\cA_{\U}$. A direct calculation shows that the two definitions of Jacobian coincide up to a nonzero constant multiple. 
\end{proof}

The above theorem yields \cite[Lemma 2.1]{FS15} which asserts that the Jacobian of the transformation $(\tau,\mathfrak{z})\mapsto \gamma \cdot (\tau,\mathfrak{z})$ for $\gamma\in \U(L)$ is 
$$
\det(\gamma)j(\gamma;\tau,\mathfrak{z})^{-(n+1)}.
$$

\begin{remark} Recall that our computations of modular forms on ball quotients represent them as power series in the variables $z_1,...,z_{n-1}$ with coefficients in some ring $M_*(\Gamma_0)$ of elliptic modular forms in the variable $\tau$, convergent on some open subset of the Siegel domain. The ring of formal power series of this type is not preserved under $\partial_{\tau}$. However, it is preserved under the \emph{Serre derivative} with respect to the variable $\tau$, $$\mathcal{S} F(\tau, \mathfrak{z}) = \frac{\partial}{\partial \tau} F(\tau, \mathfrak{z}) - \frac{k \pi i}{6} E_2(\tau) F(\tau, \mathfrak{z}),$$ where $E_2(\tau) = 1 - 24 \sum_{n=1}^{\infty} \sum_{d | n} d e^{2\pi i n \tau}$ and where $F$ has weight $k$. It is not hard to see that the Jacobian can also be written as the determinant
$$J_{\U}(F_1,...,F_{n+1}) = \left\lvert \begin{array}{cccc}
k_1F_1 & k_2F_2 & \cdots & k_{n+1}F_{n+1} \\ 
\mathcal{S} F_1 & \mathcal{S} F_2 & \cdots & \mathcal{S} F_{n+1} \\
\frac{\partial F_1}{\partial z_1} & \frac{\partial F_2}{\partial z_1} & \cdots & \frac{\partial F_{n+1}}{\partial z_1} \\ 
\vdots & \vdots & \ddots & \vdots \\ 
\frac{\partial F_1}{\partial z_{n-1}} & \frac{\partial F_2}{\partial z_{n-1}} & \cdots & \frac{\partial F_{n+1}}{\partial z_{n-1}}
\end{array}   \right\rvert,
$$ that is, as a determinant of elements of $M_*(\Gamma_0)[[z_1,...,z_{n-1}]]$.
\end{remark}

Similarly to \cite[Theorem 3.5]{Wan20},  we study the Jacobian of generators of a free algebra of unitary modular forms.
\begin{theorem}\label{th:freeJacobian}
Suppose $M_*(\Gamma)$ is a free algebra generated by forms $F_j$ of weight $k_j$ for $1\leq j \leq n+1$.  
\begin{enumerate}
\item The modular group $\Gamma$ is generated by reflections.
\item The Jacobian $J_{\U}=J_{\U}(F_1,...,F_{n+1})$ is a nonzero cusp form for $\Gamma$ of weight $n+1 + \sum_{j=1}^{n+1}k_j$ and character $\det^{-1}$. 
\item The divisor of $J_{\U}$ is the sum of mirrors $r^{\perp}$ of reflections in $\Gamma$, each with multiplicity $$\mathrm{ord}(J_{\U}, r^{\perp}) = -1 + \max\{\mathrm{ord}(\alpha): \; \sigma_{r, \alpha} \in \Gamma\}.$$
\item Let $\{ \Gamma \pi_1, ..., \Gamma \pi_s\}$ denote the $\Gamma$-equivalence classes of mirrors of reflections in $\Gamma$. Then for each $1 \le i \le s$ there exists a modular form $J_i$ for $\Gamma$ with some character (or multiplier system) and divisor $\div(J_i) = \Gamma \pi_i$. 
\end{enumerate}
\end{theorem}

\begin{proof}
(1) Since $M_*(\Gamma)$ is free,  the projective variety $(\cD_{\U} / \Gamma)^*$ is a weighted projective space and the affine cone $(\cA_{\U} / \Gamma)^*$ over $(\cD_{\U} / \Gamma)^*$ is simply the affine space $\CC^{n+1}$.  Recall that $(\cA_{\U} / \Gamma)^*$ is obtained from $\cA_{\U} / \Gamma$ by adding zero and finitely many one-dimensional cones.  Thus $\cA_{\U} / \Gamma$ is smooth and simply connected.  By a theorem of Armstrong \cite{Arm68} (cf. \cite[Theorem 3.2]{Wan20}),  $\Gamma$ is generated by elements having fixed points, and by a theorem of Gottschling \cite{Got69a} (cf. \cite[Theorem 3.1]{Wan20}), $\Gamma$ is generated by reflections.

(2) This is a consequence of Theorem \ref{th:Jacobian}.

(3) The Baily--Borel compactification $(\cD_{\U}/\Gamma)^*=\operatorname{Proj} (M_*(\Gamma))$ is identified with weighted projective space $\PP(k_1,...,k_{n+1})$ by the map
$$
[z] \mapsto [F_1(z),...,F_{n+1}(z)].
$$
Moreover,  $\cA_{\U}/\Gamma$ is an open subset of $\operatorname{Specm}(M_*(\Gamma)) \backslash \{0\}=\CC^{n+1} \backslash \{0\}$,  and $$z \mapsto (F_1(z),...,F_{n+1}(z))$$ induces a holomorphic map
$$
\pi_{\U}: \cA_{\U} \to \cA_{\U}/\Gamma \hookrightarrow \CC^{n+1}\setminus \{0\}.
$$
If $v\in \cA_{\U}$ has trivial stabilizer $\Gamma_v$, then $\pi_{\U}$ is biholomorphic on a neighborhood of $v$ since the action of $\Gamma$ on $\cA_{\U}$ is properly discontinuous.  Therefore the Jacobian $J_{\U}(v)$ is nonzero.  In particular, Gottschling's theorem implies that $J_{\U}$ is nonzero away from mirrors of reflections in $\Gamma$. Conversely, Theorem \ref{th:Jacobian} $(3)$ implies that $J_{\U}$ indeed vanishes on all mirrors of reflections in $\Gamma$.

We now calculate the multiplicity.   Suppose $r\in L$ is such that $J_{\U}$ vanishes on $r^{\perp}$.  For a generic point $x$ in the mirror $r^\perp $,  the stabilizer $\Gamma_x$ is generated by reflections $\sigma_{r, \alpha}$, $\alpha \in \mathcal{O}_\F^{\times}$.  Let $m_r$ be the order of $\Gamma_x$, and choose coordinates $(x_1, x_2,..., x_{n+1})$ at $x$ such that $r^\perp=\{x_1=0\}$. Locally at $x$, $\pi_{\U}$ has the form
$$
(x_1,x_2,..., x_{n+1}) \mapsto (x_1^{m_r},x_2,..., x_{n+1}),
$$
and we see that $J_{\U}$ vanishes along $\{x_1=0\}$ with multiplicity $m_r-1$.

(4) Let $a$ be the order of the character $\det$ on $\Gamma$. Then $J_{\U}^a\in M_*(\Gamma)$ can be expressed uniquely as a polynomial $P$ in the generators $F_j$. In $\cA_{\U}$, we have $$\div(J_{\U}^a)=a\sum_{i=1}^s c_i\Gamma \pi_i,$$ and the divisor of $J_{\U}^a$ in $\cA_{\U}/ \Gamma \hookrightarrow \CC^{n+1} \backslash \{0\}$ is $a_i\cdot\pi_i$, where $c_i$ are the multiplicities of the divisors and $a_i:=ac_i/(c_i+1)$ are integers. If $P=P_1\cdots P_t$ is the irreducible decomposition over $\CC$, then we obtain the corresponding decomposition of the zero locus in $\PP(k_1,...,k_{n+1})$: $$Z(P)=Z(P_1)\cup \cdots \cup Z(P_t).$$ In particular, each $Z(P_i)$ must correspond to a $\pi_j$ (without loss of generality, $i = j$), so the function $P_i(F_1,...,F_{n+1})$ is a modular form for $\Gamma$ with trivial character and divisor $$\mathrm{div}\, P_i(F_1,...,F_{n+1}) = ac_i \Gamma \pi_i.$$ Its $(ac_i)$-th root is the desired $J_i$.
\end{proof}

We now establish the modular Jacobian criterion for unitary modular forms,  which provides an efficient way to construct free algebras of unitary modular forms.
\begin{theorem}\label{th:modularJacobian}
Let $\Gamma$ be a finite-index subgroup of $\U(L)$.  Assume that there exist $n+1$ modular forms for $\Gamma$ with trivial character whose Jacobian is a nonzero modular form and vanishes precisely on all mirrors of reflections in $\Gamma$ with multiplicity $m-1$,  where $m$ is the maximal order of reflections in $\Gamma$ along the mirror.  Then $M_*(\Gamma) $ is freely generated by the $n+1$ forms and $\Gamma$ is generated by reflections.   
\end{theorem} 
\begin{proof}
The proof is similar to that of \cite[Theorem 5.1]{Wan20}. Due to the importance of this theorem throughout the rest of the paper, we provide a short proof for the reader's convenience.

Let $f_i$ ($1\leq i\leq n+1$) be $n+1$ unitary modular forms of weight $k_i$ whose Jacobian $J_{\U}$ satisfies the conditions of the theorem.
Suppose that $M_*(\Gamma)$ is not generated by $f_i$. We choose a modular form $f_{n+2} \in M_{k_{n+2}}(\Gamma)$ of minimal weight such that $f_{n+2} \notin \CC[f_1,...,f_{n+1}]$. For $1\leq t \leq n+2$ we define $$J_{t} = J_{\U}(f_1,..., \hat f_t, ..., f_{n+2})$$ as the Jacobian of the $n+1$ modular forms $f_i$ omitting $f_t$ (so in particular $J_{\U}=J_{n+2}$). By Theorem \ref{th:Jacobian}, the Jacobian $J_t$ is a modular form of weight $k = n+1+ \sum_{i \ne t} k_i$ and character $\det^{-1}$ on $\Gamma$. Moreover, for each reflection $\sigma_{r, \alpha} \in \Gamma$, $J_t$ vanishes on $r^\perp$ with multiplicity at least $\mathrm{ord}(\alpha)-1$. Therefore the quotient $g_t := J_t/J_{\U}$ is a holomorphic modular form in $M_*(\Gamma)$.
We compute
$$
0 = \mathrm{det} \begin{psmallmatrix} k_1 f_1 & k_2 f_2 & \cdots & k_{n+2} f_{n+2} \\ k_1 f_1 & k_2 f_2 & \cdots & k_{n+2} f_{n+2} \\ \nabla f_1 & \nabla f_2 & ... & \nabla f_{n+2} \end{psmallmatrix} = \sum_{t=1}^{n+2} (-1)^t k_t f_t J_t = \Big( \sum_{t=1}^{n+2} (-1)^t k_t f_t g_t \Big) \cdot J_{\U},
$$
and therefore
$$
(-1)^{n+1}k_{n+2}f_{n+2}= \sum_{t=1}^{n+1}(-1)^t k_t f_t g_t
$$
because $g_{n+2}=1$. In particular, each $g_t$ has weight strictly less than that of $f_{n+2}$. The construction of $f_{n+2}$ implies that $g_t \in \CC[f_1,...,f_{n+1}]$ and then $f_{n+2} \in \CC[f_1,...,f_{n+1}]$, a contradiction. Therefore $M_*(\Gamma)$ is generated by the $f_i$. By Theorem \ref{th:freeJacobian} (1), $\Gamma$ is generated by reflections.
\end{proof}

\section{Twins of free algebras of modular forms}\label{sec:twins}
In previous work \cite{Wan20a, WW20a, WW20b, WW20c} we constructed a number of free algebras of orthogonal modular forms (some new, some previously known). For each of these the Jacobian $J_{\Orth}$ of the generators is a nonzero cusp form that vanishes exactly on the mirrors of reflections in the modular group with multiplicity one. Some of the modular groups $\Gamma$ we considered are orthogonal groups of lattices $L$ with complex multiplication by the Eisenstein or Gaussian integers (or both). In this section we will show that the Jacobian criterion implies that the algebras of modular forms for the unitary groups $\Gamma \cap \U(L)$ are also freely generated.

We first describe the connection between two types of Jacobians.

\begin{proposition}\label{prop:nonzeroJ}
Let $L$ be an even Hermitian lattice of signature $(n,1)$ with $n > 1$.  Let $\Gamma$ be a finite-index subgroup of $\Orth^+(L_\ZZ)$.  Let $F_0,...,F_{2n}$ be orthogonal modular forms for $\Gamma$, with restrictions $f_0,...,f_{2n}$ to $\Gamma \cap \U(L)$. Suppose $$f_{n+1},...,f_{2n} \equiv 0$$ but the restriction of the orthogonal Jacobian to $\U(L)$ is not identically zero. Then $f_0,...,f_n$ are algebraically independent over $\CC$.
\end{proposition}

\begin{proof}
Let $J_{\Orth}$ be the orthogonal Jacobian of $F_0,...,F_{2n}$ and let $J_{\U}$ be the unitary Jacobian of $f_0,...,f_n$. We use the embedding of $\U(n,1)$ in $\Orth(2n,2)$ and the notation in \S \ref{sec:embedding}. View the $F_j$ as functions on the orthogonal tube domain about a fixed zero-dimensional cusp and write them in the form $F_j(\tau, \mathfrak{z}_{\text{hol}}, \mathfrak{z}_{\text{anti}}, w)$, such that $$f_j(\tau, \mathfrak{z}_{\text{hol}}) = F_j(\tau, \mathfrak{z}_{\text{hol}}, 0, -\overline{\zeta}).$$ By construction, $\partial_{\tau}$ and $\partial_{\mathfrak{z}_{\text{hol}}}$ span the tangent space of the Siegel domain $\cH_{\U}$ for $\U(n,1)$. Since $F_{n+1},...,F_{2n}$ vanish identically on $\cH_{\U}$, this is also true for their derivatives with respect to $\tau$ and $\mathfrak{z}_{\text{hol}}$. Moreover, the derivatives of $f_0,...,f_n$ with respect to $\tau$ and $\mathfrak{z}_{\text{hol}}$ equal the restrictions of the derivatives of $F_j$. It follows that $J_{\Orth}$ restricts to the determinant of a block matrix $$J_{\Orth}(\tau, \mathfrak{z}_{\text{hol}}, 0, -\overline{\zeta}) = \mathrm{det} \begin{pmatrix} k_0 f_0 & ... & k_n f_n & 0 & ... & 0 \\ \partial_{\tau} f_0 & ... & \partial_{\tau} f_n & 0 & ... & 0 \\ \nabla_{\mathfrak{z}_{\text{hol}}} f_0 & ...& \nabla_{\mathfrak{z}_{\text{hol}}} f_n & 0 & ... & 0 \\ \nabla_{\mathfrak{z}_{\text{anti}}}|_0 F_0 & ... & \nabla_{\mathfrak{z}_{\text{anti}}}|_0 F_n & \nabla_{\mathfrak{z}_{\text{anti}}}|_0 F_{n+1} & ... & \nabla_{\mathfrak{z}_{\text{anti}}}|_0 F_{2n} \\ \partial_w |_{-\overline{\zeta}} F_0 & ... & \partial_w |_{-\overline{\zeta}} F_n & \partial_w |_{-\overline{\zeta}} F_{n+1} & ... & \partial_w |_{-\overline{\zeta}} F_{2n}  \end{pmatrix}$$ in which $J_{\U}$ is the determinant of the upper-left block. Since $J_{\Orth}$ is not identically zero, $J_{\U}$ is also not identically zero.
\end{proof}

\begin{theorem}\label{th:twins}
Let $L$ be an even Hermitian lattice of signature $(n,1)$ over $\F=\QQ(\sqrt{-1})$ or $\QQ(\sqrt{-3})$ with $n>1$.  Let $\Gamma= \widetilde{\Orth}^+(L_\ZZ)$ or $\Orth^+(L_\ZZ)$ and define $\Gamma_{\U}:=\Gamma \cap \U(L)$.  Suppose $M_*(\Gamma)$ is freely generated by $2n+1$ orthogonal  modular forms, $n$ of whose restrictions to $\cD_{\U}$ are identically zero.  Then $M_*(\Gamma_{\U})$ is freely generated by the restrictions to $\cD_{\U}$ of the remaining $n+1$ generators.
\end{theorem}

\begin{proof}
We denote the $2n+1$ forms by $F_j$, $0 \le j \le 2n$ and their restrictions to $\cD_{\U}$ by $f_j$, and suppose $f_j = 0$ for $j \ge n+1$. Let $J_{\Orth}$ be the orthogonal Jacobian of the $F_j$. Note that $J_{\Orth}$ has its zero locus supported on Heegner divisors so its restriction to $\cD_{\U}$ is not identically zero. By Proposition \ref{prop:nonzeroJ},  the unitary Jacobian $J_{\U}$ of these $f_j$ with $0\leq j \leq n$ is also nonzero. By Theorem \ref{th:Jacobian}, $J_{\U}$ vanishes on all mirrors of reflections in $\Gamma_{\U}$; and by Lemma \ref{lem:kernelreflection}, Lemma \ref{lem:Gaussian}, Lemma \ref{lem:Eisenstein} and our discussions in \S\ref{sec:reflections}, the restriction of $J_{\Orth}$ to $\cD_{\U}$ vanishes only on mirrors of reflections in $\Gamma_{\U}$. Therefore $J_{\U}$ vanishes exactly on all mirrors of reflections in $\Gamma_{\U}$.  To finish the proof, we have to show that the orders of vanishing of $J_{\U}$ are exactly those required to apply Theorem \ref{th:modularJacobian}.

We first consider the simpler case $\Gamma= \widetilde{\Orth}^+(L_\ZZ)$. In this case, $J_{\Orth}$ vanishes precisely on all mirrors associated with $2$-roots of $L$.  When $d=-1$, by Lemma \ref{lem:kernelreflection}, the restriction of $J_{\Orth}$ vanishes precisely on all mirrors of reflections in $\widetilde{\U}(L)$ with multiplicity 2. By Theorem \ref{th:Jacobian}, $J_{\U}$ vanishes on the mirror of a reflection in $\widetilde{\U}(L)$ with multiplicity $2t-1$ for some positive integer $t$, because the order of the corresponding reflection is $2$. Thus the multiplicity of every divisor of $J_{\U}$ is one, as required. The proof when $d=-3$ is similar.

Now we consider the case $\Gamma=\Orth^+(L_{\ZZ})$. We only prove the result when $d=-3$ as $d=-1$ may be treated similarly. Let $r$ be a primitive vector of positive norm in $L$ for which $\sigma_r\in \Orth^+(L_{\ZZ})$. First suppose that the unitary reflection $\sigma_{r,\omega}$ lies in $\U(L)$, where $\omega = e^{\pi i / 3}$, so $J_{\U}$ vanishes on the mirror $r^\perp$ with multiplicity $6t-1$ for some positive integer $t$. By Lemma \ref{lem:Eisenstein}, $\sigma_{r,\omega}\in \U(L)$ if and only if $\sigma_{(1+\omega)r}\in \Orth^+(L_{\ZZ})$; since $r$ and $(1+\omega)r$ are both primitive in $L_{\ZZ}$ and their associated reflections lie in $\Orth^+(L_{\ZZ})$, the multiplicity of $r^\perp$ in the divisor of the restriction of $J_{\Orth}$ is 6. Therefore $J_{\U}$ must vanish on $r^\perp$ to multiplicity exactly $5$. Next suppose that $\sigma_{r,\omega}\not\in \U(L)$. By Lemma \ref{lem:Eisenstein}, $\sigma_{r,-\omega} \in \U(L)$ and $J_{\U}$ vanishes on the mirror $r^\perp$ with multiplicity $3t-1$ for some positive integer $t$. If the restriction of $J_{\Orth}$ vanishes on $r^\perp$ with multiplicity 3, then $J_{\U}$ vanishes on $r^\perp$ with multiplicity $2$ as claimed. Otherwise, the restriction of $J_{\Orth}$ vanishes on $r^\perp$ to order greater than 3, so there exists $a\in \mathcal{O}_{\F} \setminus \mathcal{O}_{\F}^\times$ such that $ar$ is primitive in $L_{\ZZ}$ and $\sigma_{ar}\in \Orth^+(L_{\ZZ})$. This implies that $2(L,ar)/(ar,ar)\subseteq \ZZ$ and therefore $r/(\bar{a}\latt{r,r})\in L'$. Thus $$\latt{r,r/(\bar{a}\latt{r,r})}=1/a\in \sqrt{-3}^{-1}\cdot\ZZ[\omega].$$ Since $ar$ was primitive in $L_{\ZZ}$, 
$a$ must equal $1+\omega$ up to a unit multiple. This is a contradiction because $\sigma_{r,\omega}\not\in \U(L)$.
\end{proof}

In practice, for the full unitary group it is often quite easy to find generators satisfying the conditions of Theorem \ref{th:twins} due to the following lemma.

\begin{lemma}\label{lem:zero}
Suppose $\U(L)$ admits nonzero modular forms of integral weight $k$ and trivial character. 
\begin{enumerate}
    \item If $d=-1$, then $k$ is a multiple of $4$. 
    \item If $d=-3$, then $k$ is a multiple of $6$. 
\end{enumerate}
\end{lemma}
\begin{proof}
For any unit $\alpha$ of the integers of $\QQ(\sqrt{d})$, the map $\alpha: z \mapsto \alpha\cdot z$ defines an automorphism in $\U(L)$. Suppose $\alpha$ has order $a$. For a modular form $F\in M_k(\U(L))$,
$$
F(z)=F(\alpha z)=\alpha^{-k}F(z),
$$
which (if $F$ is nonzero) implies $a | k$.
\end{proof}

By applying this to the structure results for $M_*(\Orth^+(L))$ in our previous work we immediately obtain the following free algebras of modular forms for $\U(L)$. Note that in the pairs of lattices $2U\oplus 2A_1$ -- $2U(2)\oplus 2A_1$ and $2U\oplus D_4$ -- $2U(2)\oplus D_4$ of Table \ref{tab:1full} and $2U\oplus A_2$ -- $2U(3)\oplus A_2$ and $2U\oplus D_4$ -- $2U(2)\oplus D_4$ of Table \ref{tab:3full}, there are isomorphisms of both the full orthogonal groups and the full unitary groups; and that in both tables for $L_\ZZ=2U\oplus E_8$  we have $\Orth^+(L_\ZZ)=\widetilde{\Orth}^+(L_\ZZ)$ and $\U(L)=\widetilde{\U}(L)$.

\begin{table}[ht]
\caption{Free algebras of modular forms for $\U(L)$ over $\QQ(\sqrt{-1})$}\label{tab:1full}
\renewcommand\arraystretch{1.5}
\noindent\[
\begin{array}{|c|c|c|}
\hline 
L_\ZZ & \text{weights of generators of $M_*(\Orth^+(L_\ZZ))$} & \text{weights of generators of $M_*(\U(L))$} \\ 
\hline 
2U\oplus D_4 & 4,  6,  10,  12,  16,  18,  24  & 4,  12,  16,  24 \\
\hline
2U\oplus D_6 &  4,6,8,10,12, 12,  14,16,18 & 4,  8,  12,  12,  16 \\
\hline
2U\oplus D_8 &   4,6,8, 8,  10,10,12,12,14,16,18 & 4,  8,  8,  12,  12,  16 \\
\hline 
2U\oplus E_8 & 4, 10, 12, 16, 18, 22, 24, 28, 30, 36, 42 & 4,  12,  16,  24,  28,  36 \\
\hline 
U\oplus U(2) \oplus D_4 & 2,6,8,10,12,16,20 & 8,  12,  16,  20 \\
\hline 
2U\oplus 2A_1 & 4,6,8,10,12 & 4,  8,  12 \\
\hline
2U\oplus 4A_1 & 4,4,6,6,8,10,12 & 4,  4,  8,  12 \\
\hline
U\oplus U(2)\oplus 2A_1 & 2, 4, 6, 8, 12 & 4,  8,  12 \\
\hline
\end{array} 
\]
\end{table}

\begin{table}[ht]
\caption{Free algebras of modular forms for $\U(L)$ over $\QQ(\sqrt{-3})$}\label{tab:3full}
\renewcommand\arraystretch{1.5}
\noindent\[
\begin{array}{|c|c|c|}
\hline 
L_\ZZ & \text{weights of generators of $M_*(\Orth^+(L_\ZZ))$} & \text{weights of generators of $M_*(\U(L))$} \\ 
\hline 
2U\oplus A_2 & 4, 6, 10, 12, 18 & 6, 12, 18 \\
\hline
2U\oplus D_4 & 4,  6,  10,  12,  16,  18,  24  & 6, 12, 18, 24 \\
\hline
2U\oplus E_8 & 4, 10, 12, 16, 18, 22, 24, 28, 30, 36, 42 & 12, 18, 24, 30, 36, 42 \\
\hline 
\end{array} 
\]
\end{table}

We formulate the pairs of free algebras of modular forms for the discriminant kernels $\widetilde{\Orth}^+(L_\ZZ)$ and $\widetilde{\U}(L)$ in Table \ref{tab:1kernel} and Table \ref{tab:3kernel}. In Table \ref{tab:1kernel},  the unitary Jacobian $J_{\U}$ and orthogonal Jacobian $J_{\Orth}$ are related by $$J_{\U} = \Big(J_{\Orth} \Big|_{\cD_{\U}(L)}\Big)^{1/2},$$ and in Table \ref{tab:3kernel} they are related by $$J_{\U} = \Big( J_{\Orth} \Big|_{\cD_{\U}(L)} \Big)^{2/3}.$$ To find $n$ generators of orthogonal type whose restrictions are identically zero in this case, we can often use the fact that some generators for $\widetilde{\Orth}^+(L_\ZZ)$ are in fact modular forms for $\Orth^+(L_{\ZZ})$ with a character. In this case we can use an argument similar to Lemma \ref{lem:zero} to decide when their restrictions are identically zero. As an example, for $L_{\ZZ}=2U\oplus E_6$ in Table \ref{tab:3kernel}, the generators of even weight and the squares of the generators of odd weight belong to $M_*(\Orth^+(L_{\ZZ}))$, so the restrictions of the weight 4, 7, 10, 16 generators vanish. This trick quickly solves all cases except the last lattice in Table \ref{tab:1kernel} and the last three lattices in Table \ref{tab:3kernel}. For the four cases, we compute the restrictions of generators using the Taylor expansion of unitary modular forms introduced at the end of \S \ref{sec:Taylor}. In this way, we can determine the kernel of the restriction map. We then modify the generators such that they contain $n$ forms in the kernel.  This gives a desired set of generators.

\begin{table}[ht]
\caption{Free algebras of modular forms for $\widetilde{\U}(L)$ over $\QQ(\sqrt{-1})$}\label{tab:1kernel}
\renewcommand\arraystretch{1.5}
\noindent\[
\begin{array}{|c|c|c|}
\hline 
L_\ZZ & \text{weights of generators of $M_*(\widetilde{\Orth}^+(L_\ZZ))$} & \text{weights of generators of $M_*(\widetilde{\U}(L))$} \\ 
\hline 
2U\oplus D_4 &  4,  6,  8,  8,  10,  12,  18 & 4,  8,  8,  12 \\
\hline
2U\oplus D_6 &  4,6,6,8,10,12,14,16,18 & 4,  6,  8,  12,  16 \\
\hline
2U\oplus D_8 &   4,4,6,8,10,10,12,12,14,16,18 & 4,  4,  8,  12,  12,  16 \\
\hline 
U\oplus U(2) \oplus D_4 & 2, 4, 4, 4, 4, 6, 10 & 4,  4,  4,  4 \\
\hline 
2U(2) \oplus D_4 & 2,  2,  2,  2,  2,  2,  6 & 2,  2,  2,  2 \\
\hline
\end{array} 
\]
\end{table}

\begin{table}[ht]
\caption{Free algebras of modular forms for $\widetilde{\U}(L)$ over $\QQ(\sqrt{-3})$}\label{tab:3kernel}
\renewcommand\arraystretch{1.5}
\noindent\[
\begin{array}{|c|c|c|}
\hline 
L_\ZZ & \text{weights of generators of $M_*(\widetilde{\Orth}^+(L_\ZZ))$} & \text{weights of generators of $M_*(\widetilde{\U}(L))$} \\ 
\hline 
2U\oplus A_2 &  4,6,9,10,12 & 6, 9, 12  \\
\hline
2U\oplus D_4 & 4,  6,  8,  8,  10,  12,  18 & 6,  8, 12, 18 \\
\hline 
2U\oplus E_6 & 4,6,7,10,12,15,16,18,24 & 6, 12, 15, 18, 24 \\
\hline 
U\oplus U(3)\oplus A_2 & 1, 3, 3, 3, 4 & 3, 3, 3\\
\hline
2U(2)\oplus A_2 & 2,2,2,2,3 & 2, 2, 3 \\
\hline 
2U(3)\oplus A_2 & 1, 1, 1,  1,  1 & 1, 1, 1 \\
\hline
2U(2)\oplus D_4 & 2, 2, 2, 2, 2, 2, 6 & 2, 2, 2, 6\\
\hline
\end{array} 
\]
\end{table}

Finally, we consider three exceptional cases over $\QQ(\sqrt{-3})$. When $L_\ZZ=2U\oplus E_6$,  $2U(2)\oplus A_2$,  or $U\oplus U(3)\oplus A_2$,  the algebra $M_*(\Orth^+(L_\ZZ))$ is not free. However, the algebra of modular forms for the maximal reflection subgroup $\Orth_r(L_{\ZZ})$ is free and it is easy to describe the structure of $M_*(\Orth^+(L_\ZZ))$ in terms of the larger algebra $M_*(\Orth_r(L_\ZZ))$. Using these facts, we can show that $M_*(\U(L))$ is freely generated by restrictions of the generators of $M_*(\Orth^+(L_\ZZ))$.  

\begin{theorem}\label{th:3extra}
Let $L$ be a Hermitian lattice over $\QQ(\sqrt{-3})$ whose associated $\ZZ$-lattice is  $2U\oplus E_6$,  $2U(2)\oplus A_2$,  or $U\oplus U(3)\oplus A_2$.  Then $M_*(\U(L))$ is freely generated by forms of the following weights:
\begin{align*}
& 2U\oplus E_6 & &  6, 12, 18, 24, 30 & \\
& 2U(2)\oplus A_2 & & 6, 6, 12 & \\
& U\oplus U(3)\oplus A_2 & & 6, 12, 18 &  
\end{align*}
\end{theorem}
\begin{proof}
Let $\Orth_r(L_{\ZZ})$ be the subgroup of $\Orth^+(L_\ZZ)$ generated by all reflections.  For each of these lattices,  the algebra $M_*(\Orth_r(L_{\ZZ}))$ is free and we can choose its $2n+1$ generators $f_1$, $f_2$, ...., $f_{2n-1}$, $g_1$ and $g_2$ such that
$$
M_*(\Orth^+(L_{\ZZ})) = \CC[f_1, f_2, ..., f_{2n-1}, g_1^2, g_1g_2, g_2^2].
$$ 
In all cases we find that there are $n$ generators, including one of $g_i$, whose restrictions to $\cD_{\U}(L)$ are identically zero; we label them $f_{n+1},..., f_{2n-1}, g_1$.  Applying Proposition \ref{prop:nonzeroJ} to the generators of $M_*(\Orth_r(L_{\ZZ}))$, we find that the restrictions of $f_1,...,f_{n}$ and $g_2$ are algebraically independent.

Let $J_{\U}$ be the Jacobian $$J_{\U}:=J_{\U}(f_1,...,f_{n},g_2^2)=g_2\cdot J_{\U}(f_1,...,f_{n},g_2),$$ which (by the above) is nonzero. We construct a modular form $\widetilde{J}_{\U}$ vanishing exactly on all mirrors of reflections in $\U(L)$ using restrictions of Borcherds products. (This is somewhat subtle because there are unitary reflections that do not correspond to orthogonal reflections, so some of the Borcherds products involved are not reflective; see Lemma \ref{lem:Eisenstein}.) We find that $J_{\U}$ and $\widetilde{J}_{\U}$ have the same weight.  Since $J_{\U}/\widetilde{J}_{\U}$ is a holomorphic modular form of weight zero, it is a constant.  The theorem then follows from Theorem \ref{th:modularJacobian}. 

By part (4) of Theorem \ref{th:freeJacobian}, the mirrors of unitary reflections that do not have an associated orthogonal reflection occur as the divisor of a modular forms $\widetilde{J}_0$. In all cases we observe that $g_2$ and $\widetilde{J}_0$ have the same weight; therefore, the restriction of $g_2$ is equal to $\widetilde{J}_0$ up to a scalar. 

To make this clearer, we will explain the case $L_{\ZZ}=2U\oplus E_6$ in more detail. From \cite{WW20a} we know that $\widetilde{\Orth}^+(L_{\ZZ})=\Orth_r(L_{\ZZ})$ is generated by reflections. The forms $f_1,...,f_4$ are generators of weight 6, 12, 18, 24; the forms $f_5, f_6, f_7$ are generators of weight 4, 10, 16; and the forms $g_1$ and $g_2$ are generators of weight 7 and 15 respectively.  The Jacobian $J_{\U}$ has weight $95$. The product $\widetilde{J}_{\U}$ is $h_1^{2/3}h_2^{1/3}$, where $h_1$ and $h_2$ are the restrictions to the unitary group of the Borcherds products whose input functions have the following principal parts:
\begin{align*}
&h_1:& &240e_0+q^{-1}e_0,&\\
&h_2:& &90e_0+q^{-2/3}\sum_{v}e_v, \; \text{where} \; (v,v)=4/3, \ord(v)=3.&
\end{align*}
The restriction of $g_2$ to $\cD_{\U}(L)$ is a constant multiple of $h_2^{1/3}$.
\end{proof}

\section{More free algebras of unitary modular forms}\label{sec:more}
In this section we construct free algebras of unitary modular forms which are not covered by Theorem \ref{th:twins} and therefore require considerably more effort.  Let $\widetilde{\U}_r(L)$ be the maximal reflection subgroup in $\widetilde{\U}(L)$ and let $\U_r(L)$ be the maximal reflection subgroup in $\U(L)$. The free algebras for $\widetilde{\U}_r(L)$ and $\U_r(L)$ are formulated in Table \ref{tab:1free} and Table \ref{tab:3free}. In general, $\widetilde{\U}_r(L)$ and $\U_r(L)$ need not be finite-index subgroups of $\U(L)$; however, we will show that they do have finite index in all of the cases below.

\begin{table}[ht]
\caption{Free algebras of modular forms for $\Gamma$ over $\QQ(\sqrt{-1})$}\label{tab:1free}
\renewcommand\arraystretch{1.5}
\noindent\[
\begin{array}{|c|c|c|}
\hline 
L_\ZZ & \Gamma &  \text{weights of generators of $M_*(\Gamma)$} \\ 
\hline
2U\oplus 2A_1 & \widetilde{\U}_r(L) & 4, 4,  6 \\
\hline 
2U\oplus 2A_1(2) & \U_r(L) & 2, 4, 6 \\
\hline
U\oplus U(2)\oplus 2A_1 & \widetilde{\U}_r(L) & 2,  2,  2 \\
\hline
2U\oplus D_4(2) & \U_r(L) & 4, 4, 6, 6 \\
\hline
\end{array} 
\]
\end{table}

\begin{table}[ht]
\caption{Free algebras of modular forms for $\Gamma$ over $\QQ(\sqrt{-3})$}\label{tab:3free}
\renewcommand\arraystretch{1.5}
\noindent\[
\begin{array}{|c|c|c|}
\hline 
L_\ZZ & \Gamma &  \text{weights of generators of $M_*(\Gamma)$} \\ 
\hline 
2U\oplus A_2(2) & \U(L) & 6,  6,  12 \\
\hline
2U\oplus A_2(3) & \U_r(L) & 2,  4,  6 \\
\hline
2U\oplus 2A_2 & \U(L) & 6, 12, 12, 18 \\
\hline
2U\oplus 3A_2 &  \U(L) & 6, 6, 12, 12, 18 \\
\hline
U\oplus U(3) \oplus 2A_2 & \U(L) &  6, 12, 18, 30 \\
\hline
U\oplus U(3) \oplus 3A_2 & \U(L) &  6, 12, 18, 24, 30 \\
\hline
2U\oplus D_4(2) & \U_r(L) & 6, 6, 6, 6 \\
\hline
\end{array} 
\]
\end{table}

The proofs are applications of the Jacobian criterion Theorem \ref{th:modularJacobian}, together with analysis of the quasi-pullbacks of modular forms to the group $\U(1, 1)$ (more concretely, the leading coefficients in their Taylor series on the Siegel domain). The computations of the Borcherds products were carried out in Sage \cite{sagemath}. In many cases we obtain additional algebras of unitary modular forms for other groups. When we apply the Jacobian criterion to some reflection groups, we do not know in advance that these subgroups have finite-index in $\U(L)$. This difficulty can often be overcome by a trick which is exemplified in the proof of part (1) of Theorem \ref{thm:2U+2A1}. We will now study these lattices case by case. 

\subsection{The \texorpdfstring{$2U\oplus 2A_1$}{} lattice over \texorpdfstring{$\ZZ[i]$}{}}\label{sec:2U+2A1}
The Hermitian lattice can be written as $L = H \oplus \latt{1}$; i.e. $\mathbb{Z}[i]^3$ with Gram matrix $\begin{psmallmatrix} 0 & 0 & 1/2 \\ 0 & 1 & 0 \\ 1/2 & 0 & 0 \end{psmallmatrix}$.
There are three reflective Borcherds products $F_1, F_2, F_3$ on $2U\oplus 2A_1$. Their input forms have the following principal parts at $\infty$:
\begin{align*}
&F_1:&  &8e_0+q^{-1/4}(e_{(0, 1/2,0)}+ e_{(0,i/2, 0)} );& \\
&F_2:& &68e_0+q^{-1}e_{(0, 0, 0)};& \\
&F_3:&  & 20e_0 + q^{-1/2}e_{(0, 1/2 + i/2, 0)}. & 
\end{align*}
Note that $F_2/F_1$ is holomorphic of weight $30$.  We denote the restrictions to the unitary group by $f_1$, $f_2$, $f_3$ respectively.  Each $f_i$ has only zeros of multiplicity two and therefore admits a square root; from their divisors, we see that the following modular forms are holomorphic:
\begin{align*}
&\mathrm{div}(\sqrt{f_1}) = \sum_{\substack{r \in L' \, \text{primitive} \\ \latt{r, r} = 1/4}} r^{\perp}, &\mathrm{wt}(\sqrt{f_1}) = 2;\\
&\mathrm{div}(\sqrt{f_3 / f_1}) = \sum_{\substack{r \in L' \, \text{primitive} \\ \latt{r, r} = 1/2}} r^{\perp}, &\mathrm{wt}(\sqrt{f_3 / f_1}) = 3; \\
&\mathrm{div}(\sqrt{f_2 / f_3}) = \sum_{\substack{r \in L' \, \text{primitive} \\ \latt{r, r} = 1}} r^{\perp}, &\mathrm{wt}(\sqrt{f_2 / f_3}) = 12.
\end{align*}
From the discussions in \S\ref{sec:reflections}, $f_3$ vanishes exactly on those hyperplanes $r^{\perp}$ with $\sigma_{r, i} \in \U(L)$; and $f_2/f_3$ vanishes exactly on the $r^{\perp}$ with $\sigma_{r, -1} \in \U(L)$ but $\sigma_{r, i} \not \in \U(L)$.
 

\begin{lemma} Let $\cE_4$ be the Eisenstein series of weight $4$ on $\U(L)$; that is, the restriction to $\cD_{\U}(L)$ of the orthogonal  Eisenstein series of weight $4$. Under the pull-back from $M_*(\U(L))$ to $M_*(\mathrm{SL}_2(\mathbb{Z}))$, up to nonzero multiples, $$f_1 \mapsto 0, \quad \sqrt{f_3/f_1} \mapsto  \eta^6, \quad \cE_4 \mapsto E_4.$$ In particular, $f_1$, $\sqrt{f_3/f_1}$ and $\cE_4$ are algebraically independent over $\CC$.
\end{lemma}
\begin{proof} Since $H = r^{\perp}$ for the vector $r = (0,  1/2, 0)$ of norm $\latt{r, r} = 1/4$, we see from the divisor of $f_1$ that it pulls back to zero. We also see from the divisor of $\sqrt{f_3 / f_1}$ that it does \emph{not} pull back to zero; as a modular form of weight $3$ for $\mathrm{SL}_2(\mathbb{Z})$ the pullback must be a constant times $\eta^6$. Finally, since the Eisenstein series on $\Orth^+(L_{\ZZ})$ pulls back to $$E_4(\tau)E_4(w) \in M_4(\Orth^+(H_{\ZZ})) \cong M_4(\mathrm{SL}_2(\mathbb{Z}) \times \mathrm{SL}_2(\mathbb{Z})),$$ setting $w = i$ yields \[ \cE_4(\tau, \mathfrak{z}_{\text{hol}}) = E_4(\tau) E_4(i) + O(\mathfrak{z}_{\text{hol}}). \] These forms are algebraically independent because the images $\eta^6$, $E_4$ in $M_*(\mathrm{SL}_2(\mathbb{Z}))$ are algebraically independent and because $f_1$ vanishes there.
\end{proof}

Let $\widetilde{\U}_r(L)$ be the subgroup of $\widetilde{\U}(L)$ generated by all unitary reflections, and let $\widetilde{\U}_1(L)$ be the subgroup generated only by biflections $\sigma_{r, -1} \in \widetilde{\U}(L)$ for which $\sigma_{r, i} \notin \U(L)$.

\begin{theorem}\label{thm:2U+2A1}
The following algebras of modular forms associated to $L$ are freely generated:
\begin{align*}
    M_*(\widetilde{\U}_1(L)) &= \CC[\sqrt{f_1}, \sqrt{f_3/f_1}, \cE_4] \; \text{is generated in weights} \; 2, 3, 4. \\
    M_*(\widetilde{\U}_r(L)) &= \CC[f_1,  f_3/f_1,  \cE_4] \; \text{is generated in weights} \; 4, 4, 6. \\
    M_*(\U(L)) &= \CC[f_1^2,  (f_3/f_1)^2,  \cE_4] \; \text{is generated in weights} \; 4, 8, 12.
\end{align*}
\end{theorem}

\begin{proof}
(1) Define $\Gamma_1$ as the intersection of the kernels of the characters of $\sqrt{f_1}$ and $\sqrt{f_3 / f_1}$ on the discriminant kernel $\widetilde{\U}(L)$. Since $\widetilde{\U}_1(L)$ is generated by exactly those reflections whose mirrors are not zeros of either $\sqrt{f_1}$ or $\sqrt{f_3 / f_1}$ we find that $\widetilde{\U}_1(L)$ is a subgroup of $\Gamma_1$. Thus $J_{\U_1}:=J_{\U}(\sqrt{f_1}, \sqrt{f_3/f_1}, \cE_4)$ is a nonzero modular form of weight 12 that vanishes on the mirrors of reflections in $\Gamma_1$. It follows that $J_{\U_1}/ (\sqrt{f_2/f_3})$ is a modular form of weight zero and therefore a constant. We obtain $\Gamma_1=\widetilde{\U}_1(L)$ and the ring structure from Theorem \ref{th:modularJacobian}. 
    
(2) Using (1), we see that the Jacobian $$J_{\U}(f_1, f_3/ f_1, \cE_4) = \sqrt{f_1} \cdot \sqrt{f_3 / f_1} \cdot J_{\U}(\sqrt{f_1}, \sqrt{f_3 / f_1}, \cE_4) = \text{const} \cdot \sqrt{f_2}$$ has simple zeros exactly on the mirrors of reflections in $\widetilde{\U}(L)$. The claim follows from Theorem \ref{th:modularJacobian}.  

(3) This is already covered by Theorem \ref{th:twins} (see Table \ref{tab:1full}), but we give a more direct proof here. Since $F_1^2$ and $F_3^2$ are modular forms with trivial character on $\Orth^+(2U\oplus 2A_1)$, their restrictions are modular forms with trivial character on $\U(L)$. The Jacobian is $$J_{\U}(f_1^2, (f_3/f_1)^2, \cE_4) = f_1^{3/2} (f_3 / f_1)^{3/2} J_{\U}(f_1^{1/2}, (f_3 / f_1)^{1/2}, \cE_4) = \text{const} \cdot f_3^{3/2} (f_2/f_3)^{1/2},$$ with triple zeros exactly on mirrors $r^{\perp}$ of tetraflections $\sigma_{r, i}$ and simple zeros on the mirrors of primitive biflections. The claim follows from Theorem \ref{th:modularJacobian}.
\end{proof}
From the above considerations one can also determine the ring of modular forms for the discriminant kernel. Let $\chi_1$ and $\chi_3$ be the characters of $f_1$ and $f_3/f_1$ on the discriminant kernel. Since $f_1$ has double zeros, its quasi-pullback to $\widetilde{\U}(H) = \U(H)$ has weight 6. By Lemma \ref{lem:zero} it must have a nontrivial character; therefore, $f_1$ itself has nontrivial character. On the other hand, $F_1^2$ and $F_3$ are modular forms with trivial character on $\widetilde{\Orth}^+(2U\oplus 2A_1)$, so their restrictions $f_1^2$ and $f_3$ have trivial character on $\widetilde{\U}(L)$; i.e. $\chi_1=\chi_3$ is quadratic. Altogether, we find that $M_*(\widetilde{\U}(L))$ is generated by $$f_1^2, (f_3 / f_1)^2, f_3, \mathcal{E}_4$$ modulo the obvious relation.

\begin{remark}
It is a theorem of Resnikoff--Tai \cite{RT78} that the algebra of modular forms for the group $\Gamma_{\text{RT}}:=\mathrm{SU}(2,1)\cap \mathrm{SL}(3,\ZZ[i])$ is generated by forms of weight $4$, $8$, $10$, $12$, $17$. This follows quickly from Theorem \ref{thm:2U+2A1}. If $\chi_1$, $\chi_2$ and $\chi_3$ are the characters of $\sqrt{f_1}$, $\sqrt{f_3/f_1}$ and $\sqrt{f_2/f_3}$ on $\U(L)$ respectively, then $$\Gamma_{\text{RT}} = \mathrm{ker}(\mathrm{det}) = \mathrm{ker}(\chi_1 \chi_2 \chi_3).$$ 
From the decomposition of the Jacobian into irreducibles (Theorem \ref{th:freeJacobian} (4)), we see that $\chi_1$, $\chi_2$ and $\chi_3$ generate the character group of $\U(L)$. 
By considering the values of $\chi_1,\chi_2,\chi_3$ on reflections, we determine all characters of $\U(L)$ whose kernels contain $\Gamma_{\text{RT}}$. These characters are exactly $\chi_1\chi_2\chi_3$, $\chi_1^2\chi_2^2$, $\chi_1^3\chi_2^3\chi_3$ and the identity. The first three characters correspond to modular forms $\sqrt{f_2}$, $f_3$ and $f_3 \sqrt{f_2}$ respectively.  We thus conclude that 
$$
M_*(\Gamma_{\text{RT}})=\CC[\cE_4, f_1^2, f_3, (f_3/f_1)^2, \sqrt{f_2}].
$$
It was later proved by Resnikoff--Tai \cite[Theorem 1.1]{TR81} that there are two defining relations for this ring. If the weight $k$ generator is labelled $X_k$, then one of these relations follows from the definitions:
$$
X_{10}^2=X_8X_{12};
$$
Finally observe that $X_{17}^2 / X_{10} = f_2/f_3$ lies in $M_{24}(\U(L))$ so it can be expressed uniquely as a polynomial $P$ in terms of the generators $X_4$, $X_8$, $X_{12}$. Therefore the second relation is of the form
$$
X_{17}^2-X_{10}P(X_4, X_8, X_{12})=0.
$$
\end{remark}

\subsection{The \texorpdfstring{$2U\oplus 2A_1(2)$}{} lattice over \texorpdfstring{$\ZZ[i]$}{}}\label{sec:2U+2A12}
The lattice $L = 2U \oplus 2A_1(2)$ can be realized over $\mathbb{Z}[i]$ as the Hermitian lattice $H + \mathbb{Z}[i](2)$ with Gram matrix $\begin{psmallmatrix} 0 & 0 & 1/2 \\ 0 & 2 & 0 \\ 1/2 & 0 & 0 \end{psmallmatrix}$. 
There are four reflective Borcherds products on the orthogonal group of $L$. Their input forms have the following principal parts at $\infty$.
\begin{align*}
&F_1:&  &2e_0+q^{-1/8}(e_{(0, 1/4, 0)} + e_{(0, i/4, 0)} + e_{(0, -1/4, 0)} + e_{(0, -i/4, 0)});& \\
&F_2:&  &20e_0+q^{-1/2}(e_{(0, 1/2, 0)} + e_{(0, i/2, 0)} );& \\
&F_3:&  &8e_0 + q^{-1/4}(e_{(0, (1+i)/4, 0)} + e_{(0, (1-i)/4, 0)} + e_{(0, (-1+i)/4, 0)} + e_{(0, (-1-i)/4, 0)}); & \\
&F_4:& &24e_0+q^{-1}e_{(0, 0, 0)}.& 
\end{align*}
Note that $F_2/F_1$ is holomorphic.  The restrictions to $\cD_{\U}(L)$ are labelled $f_1$, $f_2$, $f_3$, $f_4$ respectively. Then $f_3 / f_1$ and $f_2 / f_3$ are holomorphic, and:
\begin{enumerate}
\item $r^\perp \in \div(f_3)$ if and only if $\sigma_{r, i}\in \U(L)$;
\item $r^\perp \in \div(f_2/f_3)$ if and only if $\sigma_{r, -1}\in \U(L)\setminus \widetilde{\U}(L) $ but $\sigma_{r, i}\not\in \U(L)$;
\item $r^\perp \in \div(f_4)$ if and only if $\sigma_{r, -1}\in \widetilde{\U}(L)$ but $\sigma_{r, i}\not\in \U(L)$.
\end{enumerate}

Similarly to the previous case, we obtain the following pullbacks to $\U(H)$:
$$f_1 \mapsto 0, \; f_3/f_1 \mapsto  \eta^{6}, \; \cE_4 \mapsto E_4.$$ (Here $\cE_4$ is again the restriction of the orthogonal Eisenstein series to $\U(L)$.) In particular, $f_1$, $f_3/f_1$ and $\cE_4$ are algebraically independent over $\CC$.

\begin{theorem}
Let $\U_1(L)$ be the subgroup of $\U(L)$ generated by order-two reflections. Then 
\begin{align*}
M_*(\U_1(L)) =& \CC[f_1,  f_3/f_1,  \cE_4] \; \text{is generated in weights} \; 1, 3, 4. \\
M_*(\U_r(L)) =& \CC[f_1^2,  (f_3/f_1)^2,  \cE_4] \; \text{is generated in weights} \; 2, 4, 6. \\
M_*(\U(L)) =& \CC[f_1^4,  (f_3/f_1)^4, f_3^2,  \cE_4] \; \text{is generated in weights} \; 4, 4, 8, 12 \;\\ & \text{with a relation in weight} \; 16.
\end{align*}
\end{theorem}
\begin{proof} The proof follows essentially the same argument as Theorem \ref{thm:2U+2A1}.
\end{proof}

We remark that the quasi-pullback of $f_1$ to $\U(H)$ is $\eta^6$ and therefore $f_1^2 \not\in M_2(\widetilde{\U}(L))$. Thus $\U_r(L)$ does not contain $\widetilde{\U}(L)$. Indeed one can show that $M_*(\widetilde{\U}(L))$ is generated by the forms $\cE_4, f_1^4, \psi_6, f_3^2, \psi_{10}, (f_3 / f_1)^4$ of weights $4, 4, 6, 8, 10, 12$, where $\psi_6$ and $\psi_{10}$ are the unitary restrictions of Gritsenko lifts of weights $6$ and $10$ that have a quadratic character under $\Orth^+(L_{\ZZ})$. This can proven using the method of quasi-pullbacks because $\psi_6$ and $\psi_{10}$ have zeros along $\U(H)$ of order exactly $6$ resp. $2$.

\subsection{The \texorpdfstring{$U\oplus U(2)\oplus 2A_1$}{} lattice over \texorpdfstring{$\ZZ[i]$}{}}\label{sec:U+U2+2A1}
The $U \oplus U(2) \oplus 2A_1$ lattice can be realized over $\mathbb{Z}[i]$ as the Hermitian lattice $H(1+i) \oplus \langle v \rangle$ where $v$ is a vector of norm $1$, with Gram matrix $\begin{psmallmatrix} 0 & 0 & (1+i)/2 \\ 0 & 1 & 0 \\ (1-i)/2 & 0 & 0 \end{psmallmatrix}$.
There are eight reflective Borcherds products on $U\oplus U(2)\oplus 2A_1$, and their input forms have the following principal parts at $\infty$.
\begin{align*}
&F_1:&  &4e_0+q^{-1/4}(e_{(0, 1/2, 0)} + e_{(0, i/2, 0)} );& \\
&F_2:&  &4e_0+q^{-1/4}(e_{(0,1/2,(1+i)/2)}+ e_{(0,i/2,(1+i)/2)} );& \\
&F_3:&  &4e_0+q^{-1/4}(e_{((1+i)/2,1/2,0)}+ e_{((1+i)/2,i/2,0)} );& \\
&F_4:&  &8e_0+q^{-1/2}e_{(0,(1+i)/2, (1+i)/2)};& \\
&F_5:&  &8e_0+q^{-1/2}e_{((1+i)/2, 0, (1+i)/2)};& \\
&F_6:&  &8e_0+q^{-1/2}e_{((1+i)/2, (1+i)/2, 0)};& \\
&F_7:&  &12e_0+q^{-1/2}e_{(0,(1+i)/2, 0)};& \\
&F_8:&  &36e_0+q^{-1}e_{(0, 0, 0)}.&
\end{align*}
Note that $F_8/F_1F_2F_3$ is holomorphic.  We denote the restriction of $F_j$ by $f_j$. From their divisors we see that the $f_j$ have the following properties:
\begin{enumerate}
\item Each $f_j$ has a holomorphic square root.
\item $f_7=f_1f_2f_3$ and $f_8=f_1f_2f_3f_4f_5f_6$ up to scalar multiples.
\item $r^\perp \in \div(f_8)$ if and only if $\sigma_{r, i}\in \U(L)$.
\item The forms $f_1^2$, $f_2^2$, $f_3^2$, $\sqrt{f_4}$, $\sqrt{f_5}$ and $\sqrt{f_6}$ are invariant under the tetraflections whose mirrors lie in $\div(f_1f_2f_3)$.
\end{enumerate}

\begin{lemma}
The Jacobian $J_{\U}(f_1, f_2, f_3)$ is not identically zero. Equivalently, $f_1$, $f_2$ and $f_3$ are algebraically independent over $\CC$.
\end{lemma}

\begin{proof}
We apply Proposition \ref{prop:nonzeroJ} to this case. By \cite[Theorem 9.1]{WW20c}, the algebra of modular forms for the subgroup $\Gamma$ of $\Orth^+(L_{\ZZ})$ generated by the discriminant kernel and by the involution $\varepsilon$ that swaps the two copies of $A_1$ is freely generated by the modular forms $m_2, F_1^2, F_2^2, F_3^2, m_6$, where $m_2$ and $m_6$ are additive theta lifts that are modular for the full group $\Orth^+(L_{\ZZ})$. By Lemma \ref{lem:zero} the restrictions of $m_2, m_6$ to $\cD_{\U}(L)$ must be identically zero. It follows that the remaining forms $f_1^2$, $f_2^2$ and $f_3^2$ are algebraically independent.
\end{proof}

\begin{theorem}
\begin{enumerate}
\item Let $\widetilde{\U}_1(L)$ be the subgroup of $\widetilde{\U}(L)$ generated by biflections whose mirrors lie in the divisor of $f_4f_5f_6$.  Then 
$$
M_*(\widetilde{\U}_1(L)) = \CC[\sqrt{f_1}, \sqrt{f_2}, \sqrt{f_3}] \; \text{is generated in weights} \; 1, 1, 1.
$$
\item We have
\begin{align*}
M_*(\widetilde{\U}_r(L)) =& \CC[f_1, f_2, f_3] \; \text{is generated in weights} \; 2, 2, 2. \\
M_*(\widetilde{\U}(L)) =& \CC[f_1^2, f_2^2, f_3^2, f_1f_2f_3] \; \text{is generated in weights} \; 4, 4, 4, 6 \\
& \text{with a relation in weight}\; 12.
\end{align*}
\item Let $\U_1(L)$ be the subgroup of $\U(L)$ generated by tetraflections whose mirrors lie in the zero locus of $f_1f_2f_3$. Then
$$
M_*(\U_1(L)) = \CC[\sqrt{f_4}, \sqrt{f_5}, \sqrt{f_6}] \; \text{is generated in weights} \; 2, 2, 2.
$$
\item Let $\U_2(L)$ be the subgroup of $\U(L)$ generated by tetraflections whose mirrors lie in the zero locus of $f_4f_5f_6$. Then
$$
M_*(\U_2(L)) = \CC[\sqrt{f_1f_2f_3}, \cE_4, \cE_8] \; \text{is generated in weights} \; 3, 4, 8.
$$
\item Let $\U_3(L)$ be the subgroup of $\U(L)$ generated by $\U_1(L)$ and by biflections whose mirrors lie in the zero locus of $f_4f_5f_6$. Then
$$
M_*(\U_3(L)) = \CC[f_1^2, f_2^2, f_3^2] \; \text{is generated in weights} \; 4, 4, 4.
$$
\end{enumerate}
\end{theorem}

In section \S\ref{sec:twins} we saw that $M_*(\U(L))$ is also a free algebra generated by forms of weights $4, 8, 12$. Their Jacobian is $f_8^{3/2}$ of weight $27$.  This can also be proved directly by the argument below.

\begin{proof} (1) $\sqrt{f_1}, \sqrt{f_2}, \sqrt{f_3}$ are algebraically independent and modular without character on $\widetilde{\U}_1(L)$, so their Jacobian is a weight $6$ modular form that vanishes on the mirrors of biflections in $\widetilde{\U}_1(L)$. By Koecher's principle it equals $\sqrt{f_8 / (f_1 f_2 f_3)} = \sqrt{f_4 f_5 f_6}$. The claim then follows from Theorem \ref{th:modularJacobian}. Since $$J_{\U}(f_1, f_2, f_3) = \sqrt{f_1} \cdot \sqrt{f_2} \cdot \sqrt{f_3} \cdot J_{\U}(\sqrt{f_1}, \sqrt{f_2}, \sqrt{f_3}) = \sqrt{f_8}$$ has simple zeros on all mirrors we also obtain the structure of $M_*(\widetilde{\U}_r(L))$.

(2) The pullbacks of $f_2$ and $f_3$ to $\widetilde{\U}(H) = \U(H)$ are nonzero modular forms of weight $2$, so Lemma \ref{lem:zero} implies that $f_2, f_3$ have a nontrivial character on $\widetilde{\U}(L)$. Since $f_1$ is the image of $f_2$ under a tetraflection it also has nontrivial character. On the other hand, $F_1^2$, $F_2^2$, $F_3^2$ and $F_7$ are all additive lifts, so their restrictions $f_1^2$, $f_2^2$, $f_3^2$ and $f_7=f_1f_2f_3$ to $\widetilde{\U}(L)$ must have trivial character. The claim now follows from the structure of $M_*(\widetilde{\U}_r(L))$.

(5) This is also an application of Theorem \ref{th:modularJacobian}, because $$J_{\U}(f_1^2, f_2^2, f_3^2) = f_1^{3/2} f_2^{3/2} f_3^{3/2} J_{\U}(f_1^{1/2}, f_2^{1/2}, f_3^{1/2}) = f_1 f_2 f_3 f_8^{1/2}$$ has only triple zeros on the mirrors lying in the zero locus of $f_1f_2f_3$ and simple zeros on all other mirrors.

(3) We only explain why the generators are algebraically independent over $\CC$. It is clear that $f_4, f_5, f_6 \in M_4(\U_3(L))$. A direct calculation shows that they span this space. Thus $f_4$, $f_5$ and $f_6$ are algebraically independent over $\CC$.

(4) By direct calculations, we can replace the generator $\cE_{12}$ in \cite[Theorem 9.4]{WW20c} with $F_7^2$. An application of Proposition \ref{prop:nonzeroJ} to $\Orth^+(L)$ shows that the restrictions of $\cE_4$, $\cE_8$ and $F_7^2$ are algebraically independent.  We then obtain the algebraic independence of generators and the claim follows from the Jacobian criterion.
\end{proof}

\subsection{The \texorpdfstring{$2U\oplus A_2(2)$}{} lattice over \texorpdfstring{$\ZZ[\omega]$}{}}\label{sec:2U+A22}
The lattice $L = 2U \oplus A_2(2)$ is the trace form of the Hermitian lattice over $\mathbb{Z}[e^{\pi i/3}]$ with Gram matrix $\begin{psmallmatrix} 0 & 0 & 1/\sqrt{-3} \\ 0 & 2 & 0 \\ -1/\sqrt{-3} & 0 & 0 \end{psmallmatrix}.$
There are three reflective Borcherds products, and their input forms have the following principal parts.
\begin{align*}
&F_1:&  &6e_0+q^{-1/6}\sum_{v} e_v, \quad (v,v)=1/3, \; \ord(v)=6;& \\
&F_2:&  &30e_0+q^{-1/2}\sum_{u} e_u, \quad (u,u)=1,\; \ord(u)=2;& \\
&F_3:&  &24e_0 + q^{-1}e_0. & 
\end{align*}
We also need an additional, non-reflective Borcherds product:
\begin{align*}
&G:&  &36e_0+q^{-2/3}\sum_{w} e_w - q^{-1/6}\sum_{v} e_v, \quad (w,w)=4/3, \; \ord(w)=3.
\end{align*}
We denote their restrictions to $\cD_{\U}(L)$ by $f_1$, $f_2$, $f_3$, $g$ respectively.  Each of the $f_i$ and $g$ have zeros only of multiplicity three, so their cube roots are well-defined modular forms. In addition,
\begin{enumerate}
\item $f_2/f_1$ is holomorphic. 
\item $r^\perp \in \div(f_1)$ if and only if $\sigma_{r, \omega}\in \U(L)$.
\item $r^\perp \in \div(f_2/f_1)$ if and only if $\sigma_{r, -\omega}\in \U(L) \backslash \widetilde{\U}(L)$ but $\sigma_{r, \omega}\not\in \U(L)$.
\item $r^\perp \in \div(f_3)$ if and only if $\sigma_{r, -\omega}\in \widetilde{\U}(L)$ but $\sigma_{r, \omega}\not\in \U(L)$.
\item $r^\perp \in \div(g)$ if and only if $\sigma_{r, -1}\in \U(L)$ but $\sigma_{r, \omega}\not\in \U(L)$.
\end{enumerate}
 
\begin{lemma}
The forms $f_1$, $f_2/f_1$ and $\cE_6$ are algebraically independent over $\CC$. 
\end{lemma}
\begin{proof}
Let $\rho = e^{2\pi i / 3}$. Computing the Taylor expansion on the Siegel domain about $\U(H)$ yields 
\begin{align*} f_1(\tau, \mathfrak{z}_{\text{hol}}) &= \eta(\tau)^{12} \eta(\rho)^{12} \mathfrak{z}_{\text{hol}}^3 + \frac{1}{60480} \eta(\tau)^{12} E_6(\tau) \eta(\rho)^{12} E_6(\rho) \mathfrak{z}_{\text{hol}}^9 + O(\mathfrak{z}_{\text{hol}}^{15}), \\
f_2(\tau, \mathfrak{z}_{\text{hol}}) &= 3\sqrt{-3} \eta(\tau)^{36} \eta(\rho)^{36} \mathfrak{z}_{\text{hol}}^3 - \frac{3 \sqrt{-3}}{2240} \eta(\tau)^{36} E_6(\tau) \eta(\rho)^{36} E_6(\rho) \mathfrak{z}_{\text{hol}}^9 + O(\mathfrak{z}_{\text{hol}}^{15}),\\
\cE_6(\tau, \mathfrak{z}_{\text{hol}}) &= E_6(\tau) E_6(\rho) - \frac{12096}{55} \eta(\tau)^{24} \eta(\rho)^{24} \mathfrak{z}^6 + O(\mathfrak{z}^{12}).\end{align*}
In particular, the values at $\mathfrak{z}_{\text{hol}} = 0$ of $f_1, f_2/f_1, \cE_6$ are respectively $0$ and multiples of the algebraically independent forms $\Delta, E_6 \in M_*(\mathrm{SL}_2(\mathbb{Z}))$.
\end{proof}

\begin{theorem}
\begin{enumerate}
\item Let $\U_1(L)$ be the subgroup of $\U(L)$ generated by reflections whose mirrors lie in the divisor of $f_3g$.  Then 
$$
M_*(\U_1(L)) = \CC[f_1^{1/3}, (f_2/f_1)^{1/3}, \cE_6] \; \text{is generated in weights $1, 4, 6$}.
$$
\item Let $\U_2(L)$ be the subgroup of $\U(L)$ generated by $\U_1(L)$ and by order-three reflections whose mirrors lie in the divisor of $f_2$.  Then 
$$
M_*(\U_2(L)) = \CC[f_1, f_2/f_1, \cE_6] \; \text{is generated in weights $3, 6, 12$}.
$$
\item Let $\U_3(L)$ be the subgroup of $\U(L)$ generated by all reflections whose mirrors lie in the divisor of $f_2f_3$.  Then 
$$
M_*(\U_3(L)) = \CC[f_1^2, g^{1/3}, \cE_6] \; \text{is generated in weights $6, 6, 6$}.
$$
\item For the full unitary group,
$$
M_*(\U(L)) = \CC[f_1^2,  f_2/f_1,  \cE_6] \; \text{is generated in weights $6, 6, 12$}.
$$
\end{enumerate}
\end{theorem}
\begin{proof} In all cases we check that Theorem \ref{th:modularJacobian} applies.

(1) The Jacobian $$J_{\U}=J_{\U}(f_1^{1/3}, (f_2 / f_1)^{1/3}, \cE_6)$$ is nonzero of weight $14$ by the previous lemma. By definition of $\U_1(L)$,  $f_1^{1/3}, (f_2 / f_1)^{1/3}$ have trivial character, and $J_{\U}$ has at least double zeros on the mirrors of triflections and simple zeros on the mirrors of biflections in $\U_1(L)$. Therefore $J_{\U} / (f_3^{2/3} g)$ is holomorphic of weight $0$, so it is a constant.

(2) The Jacobian in this case is $$J_{\U}(f_1, f_2 / f_1, \cE_6) = f_1^{2/3} (f_2 / f_1)^{2/3} J(f_1^{1/3}, (f_2 / f_1)^{1/3}, \cE_6) = f_2^{2/3} f_3^{2/3} g$$ and it vanishes along all mirrors to the correct order.

(3) We compute the zero-value $g(\tau, 0) = \text{const} \cdot \eta(\tau)^{36}$ to see that $f_1, g, \cE_6$ are also algebraically independent. Their Jacobian has weight $21$ and must be a nonzero multiple of $f_1^{5/3} (f_2/f_1)^{2/3} f_3^{2/3}$.

(4) Finally, the Jacobian of these three forms is $$J_{\U}(f_1^2, f_2 / f_1, \cE_6) = f_1^{5/3} (f_2 / f_1)^{2/3} J_{\U}(f_1^{1/3}, (f_2 / f_1)^{1/3}, \cE_6) = f_1^{5/3} (f_2 / f_1)^{2/3} f_3^{2/3} g,$$ with zeros on mirrors of reflections of exactly the necessary order.
\end{proof}

\subsection{The \texorpdfstring{$2U\oplus A_2(3)$}{} lattice over \texorpdfstring{$\ZZ[\omega]$}{}}\label{sec:2U+A23}
There are three reflective Borcherds products on the lattice $L_{\ZZ} = 2U\oplus A_2(3)$, with the following input forms:
\begin{align*}
&F_1:&  &2e_0+q^{-1/9}\sum_{v} e_v, \quad (v,v)=2/9,\; \ord(v)=9;& \\
&F_2:&  &18e_0+q^{-1/3}\sum_{u} e_u, \quad (u,u)=2/3,\; \ord(u)=3;& \\
&F_3:&  &24e_0 + q^{-1}e_0. & 
\end{align*}
Their restrictions to $\cD_{\U}(L)$ are labelled $f_1$, $f_2$, $f_3$ respectively. They have the following properties:
\begin{enumerate}
\item $f_2/f_1$ is holomorphic. 
\item $r^\perp \in \div(f_1)$ if and only if $\sigma_{r, \omega}\in \U(L)$.
\item $r^\perp \in \div(f_2/f_1)$ if and only if $\sigma_{r, -\omega}\in \U(L) \backslash \widetilde{\U}(L)$ but $\sigma_{r, \omega}\not\in \U(L)$.
\item $r^\perp \in \div(f_3)$ if and only if $\sigma_{r, -\omega}\in \widetilde{\U}(L)$ but $\sigma_{r, \omega}\not\in \U(L)$.
\end{enumerate}

We also need the non-reflective Borcherds product
\begin{align*}
&F_4:&  &24e_0+q^{-4/9}\sum_{w} e_w - q^{-1/9}\sum_{v} e_v, \quad (w,w)=8/9,\; \ord(w)=9.
\end{align*}

Its restriction to the unitary group has only zeros of multiplicity three so it admits a holomorphic cube root $h$ of weight $4$.
 
\begin{lemma}
The forms $f_1$, $h$ and $\cE_6$ are algebraically independent over $\CC$.
\end{lemma}
\begin{proof}
Setting $\mathfrak{z}_{\text{hol}} = 0$ in the expansions of these forms on the Siegel domain sends $f_1$ to $0$ and maps $h, \cE_6$ to multiples of the algebraically independent forms $\eta^8$ and $E_6$.
\end{proof}

Using Theorem \ref{th:modularJacobian},  we prove the following result.

\begin{theorem}
We define $\U_1(L)$ as the subgroup of $\U(L)$ generated by triflections.  Then we have
\begin{align*}
    M_*(\U_1(L)) =& \CC[f_1, h, \cE_6] \; \text{is generated in weights} \; 1, 4, 6.\\
    M_*(\U_r(L)) =& \CC[f_1^2, h, \cE_6] \; \text{is generated in weights}\; 2, 4, 6. \\
    M_*(\U(L)) =& \CC[f_1^6, f_1^2h, h^3, \cE_6] \; \text{is generated in weights} \; 6, 6, 6, 12 \\ &\text{with a relation in weight $18$}.
\end{align*}
\end{theorem}

\begin{proof}
(1) By construction, $f_1$ and $h$ transform under the triflections that are not squares of hexaflections without a character, since they are nonzero on the associated mirrors (which make up the zero locus of $f_2 f_3 / f_1$). Since $f_1$ has order three zeros on the mirrors of hexaflections $\sigma_{r, \omega}$ it also has trivial character under the squares $\sigma_{r, \omega}^2$. The Jacobian of $f_1, h, \cE_6$ has weight $14$ and at least double zeros on the mirrors of all triflections, so it is a multiple of the weight $14$ form $$f_1^{2/3} (f_2 f_3 / f_1)^{2/3} = f_2^{2/3} f_3^{2/3}.$$ The claim follows from Theorem \ref{th:modularJacobian}.

(2) The Jacobian in this case is $$J(f_1^2, h, \cE_6) = f_1 J(f_1, h, \cE_6) = f_1^{5/3} (f_2 f_3 / f_1)^{2/3}$$ with the necessary orders on all mirrors. 

(3) Lemma \ref{lem:zero} implies that $f_1$ has a character $\chi_1$ of order at least $6$. Since $F_1^3$ is an Eisenstein series of weight three that transforms under $\Orth^+(L_{\ZZ})$ with a quadratic character, the order of $\chi_1$ is exactly $6$. Similarly, $h$ has a character $\chi_2$ of order at least $3$. There are three weight six Gritsenko lifts $\psi_1, \psi_2, \psi_3 \in S_6(\widetilde{\Orth}^+(L_{\ZZ}))$ which are permuted transitively by $\mathrm{Aut}(L'/L)$ and an invariant weight twelve lift $\phi \in S_{12}(\Orth^+(L_{\ZZ}))$ with $$F_4 = \psi_1^2 + \psi_2^2 + \psi_3^2 - 2(\psi_1 \psi_2 + \psi_2 \psi_3 + \psi_3 \psi_1) + \phi,$$ showing that $F_4$ has trivial character on $\Orth^+(L_{\ZZ})$, so the character $\chi_2$ of $h$ has order three. It follows that $$\mathbb{C}[f_1^6, h^3, \cE_6] \subseteq M_*(\U(L)) \subsetneq \mathbb{C}[f_1^2, h, \cE_6].$$ Since $\U(L) \ne \U_r(L)$ is not generated by reflections, its algebra of modular forms is not free, so we must have another generator of the form $f_1^a h^b$ with $0 < a < 6$ and $0 < b < 3$ and $a + 4b \equiv 0$ mod $6$. This implies $\chi_1^2 = \chi_2^{-1}$, so the missing generator is $f_1^2 h$.
\end{proof}

\subsection{The \texorpdfstring{$2U\oplus 2A_2$}{} lattice over \texorpdfstring{$\ZZ[\omega]$}{}}\label{sec:2U+2A2}
Consider the lattice $L_{\ZZ} = 2U \oplus 2A_2$ as a Hermitian lattice over $\QQ(\sqrt{-3})$. By \cite{Wan20} there are no free algebras of modular forms for any subgroups $\Gamma \le \Orth^+(L_{\ZZ})$ containing $\widetilde{\Orth}^+(L_{\ZZ})$. However, we will see that the unitary group of $L$ does admit free algebras.
There are two reflective Borcherds products on $2U\oplus 2A_2$, whose input forms have the following principal parts:
\begin{align*}
&F_1:&  &12e_0+q^{-1/3}\sum_{v} e_v, \quad (v,v)=2/3,\; \ord(v)=3;& \\
&F_2:&  &84e_0 + q^{-1}e_0. & 
\end{align*}
We also need an additional non-reflective Borcherds product:
\begin{align*}
&F_3:&  &108e_0+q^{-2/3}\sum_{w} e_w, \quad (w,w)=4/3, \ord(w)=3.
\end{align*}
Their restrictions are labelled $f_1$, $f_2$, $f_3$ respectively.  We need the following facts:
\begin{enumerate}
\item $f_2/f_1$ is holomorphic. 
\item $r^\perp \in \div(f_1)$ if and only if $\sigma_{r, \omega}\in \U(L)$.
\item $r^\perp \in \div(f_2/f_1)$ if and only if $\sigma_{r, -\omega}\in \widetilde{\U}(L)$ but $\sigma_{r, \omega}\not\in \U(L)$.
\item $r^\perp \in \div(f_3)$ if and only if $\sigma_{r, -1}\in \U(L)$ but $\sigma_{r, \omega}\not\in \U(L)$.
\end{enumerate}

Weak Jacobi forms of integral weights and lattice index $A_2$ form a free $M_*(\SL_2(\ZZ))$-module generated by forms of weight $0$, $-2$ and $-3$, denoted by $\phi_{0,A_2}$, $\phi_{-2,A_2}$, $\phi_{-3,A_2}$ (see \cite{Wir92}). We can use the Gritsenko lift to construct orthogonal modular forms from Jacobi forms (see \cite{Gri91, CG13}). Recall from \S\ref{sec:twins} that $M_*(\U(L))$ with $L_{\ZZ}=2U\oplus A_2$ is freely generated by forms of weights 6, 12, 18. As generators one can take $\cE_{A_2,6}$ and the two unitary restrictions of Gritsenko lifts $$\Grit(\Delta \phi_{0, A_2}), \; \Grit(\Delta\phi_{-3,A_2})^2.$$
We construct three Gritsenko lifts on $2U\oplus 2A_2$:
\begin{align*}
G_{12} &= \Grit(\Delta\phi_{0,A_2}\otimes\phi_{0,A_2}), \\
G_{18} &= \Grit(\Delta\phi_{-3,A_2}\otimes \phi_{0,A_2})^2 +\Grit(\Delta\phi_{0,A_2}\otimes \phi_{-3,A_2})^2,\\
G_{12}^* &= \Grit(\Delta\phi_{-3,A_2}\otimes\phi_{-3,A_2})^2.
\end{align*}
Note that $G_{12}^*=F_1^2$. Let $g_k$ be the restriction of $G_k$ to $\cD_{\U}(L)$. The pullback trick yields the following lemma. 

\begin{lemma}
The unitary modular forms $f_1$, $g_{12}$, $g_{18}$ and $\cE_6$ are algebraically independent over $\CC$.
\end{lemma}
\begin{proof} The images of $\cE_6, g_{12}, g_{18}$ under the pullback to $\U(2U \oplus A_2)$ are algebraically independent and $f_1$ is mapped to zero.
\end{proof}

\begin{theorem}
$$
M_*(\U(L)) = \CC[\cE_6, g_{12}, g_{12}^*, g_{18}] \; \text{is generated in weights} \; 6, 12, 12, 18.
$$
\end{theorem}
\begin{proof} All of the claimed generators are modular under the full unitary group $\U(L)$. Their Jacobian has weight $52$ which coincides with the weight of $f_1^{5/3} (f_2 / f_1)^{2/3} f_3^{1/3}$. The claim follows from Theorem \ref{th:modularJacobian} similarly to the previous cases.
\end{proof}

The product $F_3$ splits in $\widetilde{\Orth}^+(L_{\ZZ})$ as $H_1 \cdot H_2$ where $H_1, H_2$ are products of weight $27$. Let $h_1, h_2$ be their restrictions to the unitary group. One of these (say, $h_1$) vanishes on the mirror of the reflection $\sigma$ that exchanges the two copies of $A_2$, and the other does not. Let $\U_1(L) \le \U(L)$ be the subgroup generated by the discriminant kernel and by $\sigma$. Then
$$
M_*(\U_1(L))=\CC[f_1, \cE_6, h_2^{1/3}, g_{12}] \; \text{is generated in weights} \; 6, 6, 9, 12,
$$
and the Jacobian of the generators is $$J_{\U}(f_1, \cE_6, h_2^{1/3}, g_{12}) = f_2^{2/3} h_1^{1/3}.$$ Note that $h_1^{1/3}$ and $h_2^{1/3}$ are necessarily the unitary restrictions of the Gritsenko lifts of weight $9$ that are skew-symmetric resp. symmetric under swapping the two copies of $A_2$: up to multiples, $$h_1^{1/3}, h_2^{1/3} = \Grit(\Delta \cdot (\phi_{-3,A_2} \otimes \phi_{0, A_2} \mp \phi_{0, A_2} \otimes \phi_{-3, A_2}) ).$$ Since any modular form in $M_*(\widetilde{\U}(L))$ can be split into parts that are symmetric and anti-symmetric under $\sigma$, we find that $$M_*(\widetilde{\U}(L)) = M_*(\U_1(L)) \oplus h_1^{1/3} M_*(\U_1(L))$$ is generated in weights $6,6,9,9,12$ with a single relation in weight $18$.

\subsection{The \texorpdfstring{$2U\oplus 3A_2$}{} lattice over \texorpdfstring{$\ZZ[\omega]$}{}}\label{sec:2U+3A2}
Similarly, by \cite{Wan20} there are no free algebras of modular forms for any subgroups $\Gamma \le \Orth^+(L_{\ZZ})$ containing $\widetilde{\Orth}^+(L_{\ZZ})$. However, there are free algebras for some subgroups of $\U(L)$. There are two reflective Borcherds products for $L_{\ZZ}$: 
\begin{align*}
&F_1:&  &6e_0+q^{-1/3}\sum_{v} e_v, \quad (v,v)=2/3,\; \ord(v)=3;& \\
&F_2:&  &78e_0 + q^{-1}e_0. & 
\end{align*}
We also need an additional non-reflective Borcherds product:
\begin{align*}
&F_3:&  &180e_0+q^{-2/3}\sum_{w} e_w, \quad (w,w)=4/3,\; \ord(w)=3.
\end{align*}
We label their unitary restrictions $f_1$, $f_2$, $f_3$ respectively. Similarly to the previous subsection, we have
\begin{enumerate}
\item $f_2/f_1$ is holomorphic. 
\item $r^\perp \in \div(f_1)$ if and only if $\sigma_{r, \omega}\in \U(L)$.
\item $r^\perp \in \div(f_2/f_1)$ if and only if $\sigma_{r, -\omega}\in \widetilde{\U}(L)$ but $\sigma_{r, \omega}\not\in \U(L)$.
\item $r^\perp \in \div(f_3)$ if and only if $\sigma_{r, -1}\in \U(L)$ but $\sigma_{r, \omega}\not\in \U(L)$.
\end{enumerate}

We construct three forms on $2U\oplus 3A_2$ as Gritsenko lifts (or their squares):
\begin{align*}
G_{12} =& \Grit(\Delta \phi_{0, A_2} \otimes \phi_{0, A_2} \otimes \phi_{0, A_2}), \\
G_{12}^* =& \Grit(\Delta \phi_{-3, A_2} \otimes \phi_{-3, A_2} \otimes \phi_{0, A_2})^2 + \Grit(\Delta \phi_{-3, A_2} \otimes \phi_{0, A_2} \otimes \phi_{-3, A_2})^2  \\
& + \Grit(\Delta \phi_{0, A_2} \otimes \phi_{-3, A_2} \otimes \phi_{-3, A_2})^2, \\
G_{18} = & \Grit(\Delta \phi_{-3, A_2} \otimes \phi_{0, A_2} \otimes \phi_{0, A_2})^2 + \Grit(\Delta \phi_{0, A_2} \otimes \phi_{-3, A_2}  \otimes \phi_{0, A_2})^2 \\
& + \Grit(\Delta \phi_{0, A_2} \otimes \phi_{0,  A_2} \otimes \phi_{-3, A_2})^2.
\end{align*}
Note that $F_1$ is also a Gritsenko lift of the singular-weight Jacobi form $\Delta \phi_{-3, A_2} \otimes \phi_{-3, A_2} \otimes \phi_{-3, A_2}$. Moreover, $F_1^2$ is modular without character on $\Orth^+(2U\oplus 3A_2)$.  We denote by $g_k$ the restriction of $G_k$ to $\cD_{\U}(L)$.

\begin{lemma}
The unitary modular forms $f_1, g_{12}, g_{12}^*, g_{18}, \cE_6$ are algebraically independent over $\CC$.
\end{lemma}
\begin{proof} The pullbacks of $g_{12}, g_{12}^*, g_{18}, \cE_6$ to the unitary group of $2U \oplus 2A_2$ over $\ZZ[\omega]$ generate its ring of modular forms, and $f_1$ has a zero there.
\end{proof}

\begin{theorem}
$$
M_*(\U(L)) = \CC[f_1^2, \cE_6, g_{12}, g_{12}^*, g_{18}] \; \text{is generated in weights} \; 6,6,12,12,18.
$$
\end{theorem}
\begin{proof} By construction, the Eisenstein series $\cE_6$ and the Gritsenko lifts are modular under $\Orth^+(L_{\ZZ})$. Thus these generators are modular without character on $\U(L)$. The Jacobian $$J_{\U} = J_{\U}(f_1^2, \cE_6, g_{12}, g_{12}^*, g_{18})$$ has weight 59 and at least fifth-order zeros on the mirrors of hexaflections, double zeros on the mirrors of triflections and simple zeros on biflections of $L$. By Koecher's principle it coincides with $f_1^{5/3} (f_2/f_1)^{2/3} f_3^{1/3}$. Theorem \ref{th:modularJacobian} implies the desired result.
\end{proof}

Similarly to the case $2U\oplus 2A_2$, we define $\U_1(L)$ as the subgroup of $\U(L)$ generated by the discriminant kernel and by the permutations of three copies of $A_2$. There is a modular form $h$ of weight $45$ which is a factor of $F_3$ and vanishes precisely along mirrors of reflections corresponding to these permutations.  Then the algebra of modular forms for $\U_1(L)$ is freely generated by forms of weight 3, 6, 6, 9, 12. The Jacobian is $f_2^{2/3}h^{1/3}$. As generators we can take $f_1$, $\cE_6$, $g_{12}$ and the restrictions of Gritsenko lifts
\begin{align*}
&\Grit(\Delta(\phi_{-3,A_2}\otimes\phi_{-3,A_2}\otimes\phi_{0,A_2}+\phi_{-3,A_2}\otimes\phi_{0,A_2}\otimes\phi_{-3,A_2}+\phi_{0,A_2}\otimes\phi_{-3,A_2}\otimes\phi_{-3,A_2})),\\
&\Grit(\Delta(\phi_{-3,A_2}\otimes\phi_{0,A_2}\otimes\phi_{0,A_2}+\phi_{0,A_2}\otimes\phi_{-3,A_2}\otimes\phi_{0,A_2}+\phi_{0,A_2}\otimes\phi_{0,A_2}\otimes\phi_{-3,A_2})).
\end{align*}

\subsection{The \texorpdfstring{$U\oplus U(3)\oplus 2A_2$}{} lattice over \texorpdfstring{$\ZZ[\omega]$}{}}\label{sec:U+U3+2A2}
We need three Borcherds products on $U\oplus U(3) \oplus 2A_2$: 
\begin{align*}
&F_1:&  &30e_0+q^{-1/3}\sum_{v} e_v, \quad (v,v)=2/3, \;\ord(v)=3;& \\
&F_2:&  &30e_0 + q^{-1}e_0; &\\
&F_3:&  &270e_0+q^{-2/3}\sum_{u} e_u, \quad (u,u)=4/3,\; \ord(u)=3.&
\end{align*}
We remark that $F_1^3$ is a product of fifteen distinct holomorphic Borcherds products of weight 3 over $\widetilde{\Orth}^+(L_{\ZZ})$. These span a five-dimensional space that coincides exactly with the space of weight 3 additive theta lifts. In addition $F_3$ is a product of fifteen holomorphic Borcherds products of weight 9. The unitary restrictions of $F_j$ are labelled $f_j$.  We have the following facts:
\begin{enumerate}
\item $f_1 = f_2$.
\item $r^\perp \in \div(f_1)$ if and only if $\sigma_{r, \omega}\in \U(L)$.
\item $r^\perp \in \div(f_3)$ if and only if $\sigma_{r, -1}\in \U(L)$ but $\sigma_{r, \omega}\not\in \U(L)$.
\end{enumerate}

\begin{remark}
It was proved by Freitag--Salvati-Manni \cite{FS14} that $\widetilde{\U}(L)$ is the subgroup of $\U(L)$ generated by triflections and that $M_*(\widetilde{\U}(L))$ is generated by the unitary restrictions of the fifteen Borcherds products of weight 3, and that these forms continue to span a 5-dimensional space.
\end{remark}

We will consider the ring of modular forms for the full unitary group $\U(L)$.

\begin{lemma}
The forms $f_1$, $\cE_6$, $\cE_{12}$ and $\cE_{18}$ are algebraically independent over $\CC$.
\end{lemma}
\begin{proof} The pullbacks of $\cE_6$, $\cE_{12}$ and $\cE_{18}$ to the unitary group of $U \oplus U(3) \oplus A_2$ can be computed explicitly; they are algebraically independent invariants of degrees $2, 4, 6$ of $\mathrm{Aut}(L'/L)$  acting on the three Eisenstein series of weight $3$ in $M_*(\widetilde{\U}(U \oplus U(3) \oplus A_2))$ that generate the latter ring. The product $f_1$ vanishes on $U \oplus U(3) \oplus A_2$, and the claim follows.
\end{proof}

\begin{theorem}
$$
M_*(\U(L)) = \CC[f_{1}^2, \cE_6, \cE_{12}, \cE_{18}] \; \text{is generated in weights} \; 6, 12, 18, 30.
$$
\end{theorem}
\begin{proof} Note that $f_1$ admits a holomorphic cube root. It is clear that on the reflection group $\U_r(L)$, $\sqrt[3]{f_1}$ has a character of order at most $6$ and therefore that $f_1^2 \in M_*(\U_r(L))$. The Jacobian $J_{\U}(f_1^2, \cE_6, \cE_{12}, \cE_{18})$ is a modular form of weight $70$ that vanishes on mirrors of reflections to at least the orders of the weight $70$ form $f_1^{5/3} f_3^{1/3}$, so these are equal up to a scalar and the claim follows from Theorem \ref{th:modularJacobian}. Now the smaller ring $M_*(\U(L))$ must be of the form $$M_*(\U(L)) = \mathbb{C}[f_1^a, \cE_6, \cE_{12}, \cE_{18}]$$ for some $a \ge 2$, since the Eisenstein series are already modular under $\U(L)$ by construction. In particular $M_*(\U(L))$ is free, so $\U(L) = \U_r(L)$ is already generated by reflections.
\end{proof}

\begin{remark} As a set of generators for $M_*(\U(L))$ one can also take the Eisenstein series of weights $6, 12, 18, 30$. We will need to use this fact in the next subsection. To prove this it suffices to show that $\cE_{30}$ cannot be written as a polynomial in $\cE_6, \cE_{12}, \cE_{18}$. We considered the Taylor expansions of the Eisenstein series about the sublattice $U \oplus U(3)$, which are power series in two variables $\mathfrak{z}_{\text{hol}} = (z_1, z_2)$ with coefficients in the ring $$M_*(\U(U \oplus U(3))) \cong M_{6*}(\Gamma_0(3)^+) = \mathbb{C}[e_3^2, s_6],$$ where $\Gamma_0(3)^+ = \langle \Gamma_0(3), \tau \mapsto -1/3\tau \rangle$ is the Fricke group, and $e_3$ and $s_6$ are the forms $$e_3(\tau) = 2 \Big( 1 - 9 \sum_{n=1}^{\infty} \sum_{d | n} d^2 \left( \frac{-3}{d}\right) q^n \Big) - \Big( 1 + 6 \sum_{n=1}^{\infty} \sum_{d | n} \left(\frac{-3}{d}\right) q^n \Big)^3 = 1-36q - 54q^2 - ...$$ and $s_6(\tau) = \eta(\tau)^6 \eta(3\tau)^6 = q - 6q^2 + 9q^3 \pm ...$. Let $\Omega := s_6(\lambda)^{1/6}$ for the CM point $\lambda = \frac{-3 + \sqrt{-3}}{6}$, such that $e_3(\lambda)^2 = -108 \Omega^6$. After rescaling the Taylor expansions by $\Omega^{-k}$ the coefficients lie in $\mathbb{Q}[e_3^2, s_6]$ (although the rational numbers that appear in the Eisenstein series of higher weights are rather complicated); for example, $$\cE_6\left(\tau, \frac{z_1}{2\pi i}, \frac{z_2}{2\pi i}\right) = -8 \Omega^6  \Big( e_3^2 + 72s_6 - \frac{126}{5} s_6^2 (z_1^6 + z_2^6) + ... \Big)$$ To see that $\cE_{30}$ is not a polynomial in $\cE_6, \cE_{12}, \cE_{18}$ it is sufficient to compute the Taylor coefficients up to and including degree $12$. The power series expansions are attached as an auxiliary file. We also include the Eisenstein series $\cE_{24}$ as a correctness check for the algorithm; this is indeed uniquely expressable as a polynomial in $\cE_6, \cE_{12}, \cE_{18}$.

\end{remark}

\subsection{The \texorpdfstring{$U\oplus U(3)\oplus 3A_2$}{} lattice over \texorpdfstring{$\ZZ[\omega]$}{}} \label{sec:U+U3+3A2}
We will need three Borcherds products on $L_{\ZZ} = U\oplus U(3) \oplus 3A_2$, whose input forms have the following principal parts at $\infty$:
\begin{align*}
&F_1:&  &24e_0+q^{-1/3}\sum_{v} e_v, \quad (v,v)=2/3,\; \ord(v)=3;& \\
&F_2:&  &24e_0 + q^{-1}e_0; &\\
&F_3:&  &450e_0+q^{-2/3}\sum_{u} e_u, \quad (u,u)=4/3,\; \ord(u)=3.&
\end{align*}
The unitary restrictions of $F_j$ are labelled $f_j$.  The following facts are easy to check:
\begin{enumerate}
\item $f_1 = f_2$.
\item $r^\perp \in \div(f_{1})$ if and only if $\sigma_{r, \omega}\in \U(L)$.
\item $r^\perp \in \div(f_3)$ if and only if $\sigma_{r, -1}\in \U(L)$ but $\sigma_{r, \omega}\not\in \U(L)$.
\end{enumerate}

\begin{theorem}
The algebra $M_*(\U(L))$ is freely generated by modular forms of weights $6$, $12$, $18$, $24$, $30$. 
\end{theorem}
\begin{proof} There is a natural pullback map $$\varphi : M_*(\Orth^+(L_{\ZZ})) \longrightarrow M_*(\Orth^+(U \oplus U(3) \oplus 2A_2))$$ which induces a pullback map $$\varphi_{\U} : M_*(\U(L)) \longrightarrow M_*(\U(U \oplus U(3) \oplus 2A_2)).$$ First we note that the image of the Eisenstein series $\cE_6$ under $\varphi$ differs from the Eisenstein series on $U \oplus U(3) \oplus 2A_2$ by a cusp form. However, the structure theorem of the previous section shows that all $\Orth^+$-modular cusp forms of weight $6$ on $U \oplus U(3) \oplus 2A_2$ have vanishing unitary restriction. Therefore the unitary Eisenstein series $\cE_6$ on $U \oplus U(3) \oplus 3A_2$ is mapped exactly to $\cE_6$ under $\varphi_{\U}$.

We recall that $\cE_k$ is defined as the additive theta lift of the vector-valued Eisenstein series $E_{k+2 - \mathrm{rk}(L_{\ZZ})/2, 0}$ and that the pullback $\varphi$ commutes with a map called \emph{theta-contraction} $\theta$ on vector-valued modular forms (cf. \cite{Ma19}). It is not difficult to find preimages under $\theta$ of the Eisenstein series $E_{k-2, 0}$, $k \in \{12, 18, 30\}$ as linear combinations of vector-valued modular forms of the form $$t_{k, a, b} := E_6^a \cdot  \Big(\vartheta^{3b} E_{k - 3 - 6(a+b), 0} \Big),$$ where $E_6$ denotes the scalar Eisenstein series of weight $6$ and where $\vartheta$ is the componentwise Serre derivative; and since the action of $\Orth^+(L)$ on $L'/L$ commutes with $\vartheta$, the additive theta lifts of $t_{k, a, b}$ are modular forms for $\Orth^+(L_{\ZZ})$. By passing to the unitary restrictions we obtain preimages $X_6, X_{12}, X_{18}, X_{30} \in M_*(\U(L))$ under $\varphi_{\U}$ of the generators of $M_*(\U(U \oplus U(3) \oplus 2A_2))$.

Now $f_1^2$ has a trivial character under all reflections, so $$X_6, f_1^2, X_{12}, X_{18}, X_{30} \in M_*(\U_r(L)).$$ These forms are algebraically independent (because the pullbacks to $U \oplus U(3) \oplus 2A_2$ of $X_k$ are algebraically independent and $f_1$ is sent to zero), and their Jacobian has weight $95$, so it is a constant multiple of $f_1^{5/3} f_3^{1/3}$. From Theorem \ref{th:modularJacobian} we conclude $$M_*(\U_r(L)) = \mathbb{C}[X_6, X_{12}, X_{18}, f_1^2, X_{30}].$$ The subring $M_*(\U(L))$ is of the form $\mathbb{C}[X_6, X_{12}, X_{18}, f_1^{2a}, X_{30}]$ for some $a \in \mathbb{N}$ and is therefore freely generated, so $\U(L) = \U_r(L)$ is already generated by reflections.
\end{proof}

\begin{remark}
It was proved in \cite{AF02, Fre02} that $M_*(\widetilde{\U}(L))$ is generated by ten additive theta lifts of singular weight 3 that satisfy 270 cubic relations.
\end{remark}

\subsection{The \texorpdfstring{$2U\oplus D_4(2)$}{} lattice}
The root lattice $L_{\ZZ} = 2U \oplus D_4(2)$ can be realized as a Hermitian lattice over both $\mathbb{Z}[i]$ (for example with Gram matrix $\begin{psmallmatrix} 0 & 0 & 0 & 1/2 \\ 0 & 2 & 1+i & 0 \\ 0 & 1-i & 2 & 0 \\ 1/2 & 0 & 0 & 0 \end{psmallmatrix}$) and $\mathbb{Z}[e^{\pi i / 3}]$ (with Gram matrix $\begin{psmallmatrix} 0 & 0 & 0 & \lambda \\ 0 & 2 & 2\lambda & 0 \\ 0 & 2\overline{\lambda} & 2 & 0 \\ \overline{\lambda} & 0 & 0 & 0 \end{psmallmatrix}$ with $\lambda = (3 + \sqrt{-3}) / 6$).

There are three reflective Borcherds products on $2U\oplus D_4(2)$:
\begin{align*}
&F_1:& &12e_0+q^{-1/4}\sum_{v}e_v, \quad (v,v)=1/2, \ord(v)=4;&\\
&F_2:& &48e_0+q^{-1/2}\sum_{v}e_u, \quad (u,u)=1, \ord(u)=2;&\\
&F_3:& &24e_0+q^{-1}e_0.&
\end{align*}
In both cases below the unitary restriction of $F_i$ is labelled $f_i$.

\subsubsection{$2U \oplus D_4(2)$ over \texorpdfstring{$\ZZ[i]$}{}}\label{sec:2U+D42}
The forms $f_i$ have the following properties:
\begin{enumerate}
    \item $f_2/f_1$ is holomorphic;
    \item $r^\perp \in \div(f_1)$ if and only if $\sigma_{r, i}\in \U(L)$;
    \item $r^\perp \in \div(f_3f_2/f_1)$ if and only if $\sigma_{r, -1}\in \U(L)$ but $\sigma_{r, i}\not\in \U(L)$.
\end{enumerate}
By \cite[\S 3.5]{Wan20a}, the algebra of modular forms for the maximal reflection subgroup in the integral orthogonal group of $2U\oplus D_4(2)$ is freely generated by three forms of weight 4 and four forms of weight 6. The generators can be constructed as the additive lifts $\cE_4$ and $\cE_6$ of Jacobi Eisenstein series of weight $4$ and $6$, and five basic forms whose first Fourier--Jacobi coefficients are the generators of the algebra of weak Jacobi forms associated with the root system $F_4$ invariant under the Weyl group (see \cite{Wir92}). We label them as follows:
\begin{align*}
F_{6,1} &=\eta^{12}\phi_{0,F_4,1}\cdot \xi^{1/2} + O(\xi); \\
F_{4,1} &=\eta^{12}\phi_{-2,F_4,1}\cdot \xi^{1/2} + O(\xi); \\
F_{6,2} &=\eta^{24}\phi_{-6,F_4,2}\cdot \xi + O(\xi^2); \\
F_{4,2} &=\eta^{24}\phi_{-8,F_4,2}\cdot \xi + O(\xi^2); \\
F_{6,3} &=\eta^{36}\phi_{-12,F_4,3}\cdot \xi^{3/2} + O(\xi^2).
\end{align*}
Note that $F_{6,1}F_{4,1}$, $F_{6,2}$ and $\cE_6$ are modular forms with trivial character for the full orthogonal group. By Lemma \ref{lem:zero} their restrictions to $\U(3,1)$ over $\QQ(\sqrt{-1})$ are identically zero. In particular one of the restrictions of $F_{6,1}$ and $F_{4,1}$ is identically zero; after a computation, we find that the restriction of $F_{4,1}$ is identically zero and that of $F_{6, 1}$ is not.
Applying Proposition \ref{prop:nonzeroJ} to this case, 
we conclude that $J_{\U}(\cE_4, F_{6,1}, F_{4,2}, F_{6,3})$ is not identically zero so it must equal $f_1f_2^{1/2}f_3^{1/2}$ up to scalar. By applying Theorem \ref{th:modularJacobian}, we obtain the algebra structure:
\begin{theorem}\label{th:2U+D421}
The algebra of modular forms for $\U_r(L)$ is freely generated by forms of weight $4$, $4$, $6$, $6$; and the algebra of modular forms for $\U(L)$ is generated by forms of weight $4$, $4$, $12$, $12$, $12$ with a relation in weight $24$. 
\end{theorem}

\subsubsection{$2U \oplus D_4(2)$ over \texorpdfstring{$\ZZ[\omega]$}{}}\label{sec:2U+D423}
In this case, $\div(f_1 f_2 f_3)$ consists exactly of the mirrors $r^{\perp}$ where $\sigma_{r, -\omega} \in \U(L)$ but $\sigma_{r, \omega} \not\in\U(L)$. The forms $\cE_4, F_{4, 1}^2$ and $F_{4, 2}$ from above are modular forms with trivial character under the full orthogonal group whose weights are not multiples of six, which implies that their restrictions to $\U(3,1)$ over $\QQ(\sqrt{-3})$ are identically zero. We apply Proposition \ref{prop:nonzeroJ} to see that $J_{\U}(\cE_6, F_{6,1}, F_{6,2}, F_{6,3})$ is not identically zero and equals $f_1^{2/3}f_2^{2/3}f_3^{2/3}$ up to scalar. By Theorem \ref{th:modularJacobian}, we obtain:
\begin{theorem}
The algebra of modular forms for $\U_r(L)$ is freely generated by $4$ forms of weight $6$; and the algebra of modular forms for $\U(L)$ is generated by forms of weight $6$, $6$, $12$, $12$, $12$ with a relation in weight $24$. 
\end{theorem}

\subsection{The \texorpdfstring{$2U(2)\oplus 2A_1$}{} lattice over \texorpdfstring{$\ZZ[i]$}{}}\label{sec:2U2+2A1}
The root lattice $L_{\ZZ} = 2U(2) \oplus 2A_1$ is the trace form of the Gaussian lattice $L$ with Gram matrix $\begin{psmallmatrix} 0 & 0 & 1 \\ 0 & 1 & 0 \\ 1 & 0 & 0 \end{psmallmatrix}$. There are ten Borcherds products of singular weight $1$ on $L$, and their input functions have the following principal parts at $\infty$:
\begin{align*}
&F_1:& &2 e_0 + q^{-1/4} e_{(0, 1/2, 0)} + q^{-1/4} e_{(0, i/2, 0)}; \\
&F_2:& &2 e_0 + q^{-1/4} e_{(0, 1/2, (1+i)/2)} + q^{-1/4} e_{(0, i/2, (1+i)/2)}; \\
&F_3:& &2 e_0 + q^{-1/4} e_{((1+i)/2, 1/2, 0)} + q^{-1/4} e_{((1+i)/2, i/2, 0)}; \\
&F_4:& &2 e_0 + q^{-1/4} e_{((1+i)/2, 1/2, (1+i)/2)} + q^{-1/4} e_{((1+i)/2, i/2, (1+i)/2)}; \\
&G_1:& &2 e_0 + q^{-1/4} e_{(0, 1/2, 1/2)} + q^{-1/4} e_{(0, i/2, 1/2)}; \\
&G_2:& &2 e_0 + q^{-1/4} e_{(1/2, 1/2, 0)} + q^{-1/4} e_{(1/2, i/2, 0)}; \\
&G_3:& &2 e_0 + q^{-1/4} e_{(1/2, 1/2, i/2)} + q^{-1/4} e_{(1/2, i/2, i/2)}; \\
&G_4:& &2 e_0 + q^{-1/4} e_{(0, 1/2, i/2)} + q^{-1/4} e_{(0, i/2, i/2)}; \\
&G_5:& &2 e_0 + q^{-1/4} e_{(i/2, 1/2, 0)} + q^{-1/4} e_{(i/2, i/2, 0)}; \\
&G_6:& &2 e_0 + q^{-1/4} e_{(i/2, 1/2, 1/2)} + q^{-1/4} e_{(i/2, i/2, 1/2)}. 
\end{align*}
When we consider the restrictions $f_j, g_j$ to the unitary group, we find that $f_1,...,f_4$ have only zeros of multiplicity two and therefore admit holomorphic square roots, and that $g_j$ have only simple zeros. Moreover, $g_1 = g_4$, $g_2 = g_5$ and $g_3 = g_6$. 
There is also a reflective Borcherds product $F$ of weight $4$, whose input form has principal part $8 e_0 + q^{-1/2} e_{(0, (1+i)/2, 0)}$, and whose associated reflection exchanges the two copies of $A_1$. We denote its unitary restriction by $f$; then $f = f_1 f_2 f_3 f_4$. The divisor $\div(f g_1 g_2 g_3)$ consists exactly of the mirrors $r^{\perp}$ of tetraflections $\sigma_{r, i} \in \U(L)$, and the square of any such tetraflection lies in $\widetilde{\U}(L)$. Since the divisors of $g_i$ are fixed by all tetraflections, the $g_i$ are modular forms (with character) for the group $$\Gamma = \langle \widetilde{\U}(L), \sigma_{r, i}: \; r^{\perp} \in \div(f g_1 g_2 g_3)\rangle;$$ moreover, $g_i$ transform without character under the tetraflections associated to $\div(f)$.

\begin{lemma}
The modular forms $g_1$, $g_2$ and $g_3$ are algebraically independent over $\CC$. Any three forms among $f_1$, $f_2$, $f_3$ and $f_4$ are algebraically independent over $\CC$.
\end{lemma}
\begin{proof}
By Theorem 4.3 and Corollary 4.4 of \cite{WW20b}, the forms $F_1,...,F_4,G_1,...,G_6$ satisfy five four-term quadratic relations, and any five among them whose squares are linearly independent are already algebraically independent. The claim now follows from Proposition \ref{prop:nonzeroJ}.
\end{proof}

Since $J_{\U}(g_1,g_2,g_3)$ has weight $6$ and triple zeros along the tetraflections with mirrors in $\div(f)$, it equals $f^{3/2}$ up to scalar. Let $\U_f(L)$ be the subgroup of $\Gamma$ generated by tetraflections associated with $\div(f)$ and let $\U_{0}(L)$ be the subgroup of $\Gamma$ generated by $\U_f(L)$ and by the biflections (but not tetraflections) associated to $\div(g_1g_2g_3)$. Then we have 
\begin{theorem}
\begin{align*}
M_*(\Gamma)&=\CC[g_1^4, g_2^4, g_3^4] \; \text{is generated in weights} \; 4, 4, 4,\\
M_*(\U_{0}(L))&=\CC[g_1^2, g_2^2, g_3^2] \; \text{is generated in weights} \; 2, 2, 2,\\
M_*(\U_f(L))&=\CC[g_1, g_2, g_3] \; \text{is generated in weights} \; 1, 1, 1.
\end{align*}
\end{theorem}
\begin{proof} All of the claims follow from Theorem \ref{th:modularJacobian} after observing that the Jacobians $$J_{\U}(g_1, g_2, g_3) = f^{3/2}, \; J_{\U}(g_1^2, g_2^2, g_3^2) = f^{3/2} g_1 g_2 g_3, \; J_{\U}(g_1^4, g_2^4, g_3^4) = f^{3/2} g_1^3 g_2^3 g_3^3$$ have zeros of the necessary orders along the mirrors of reflections in the respective modular groups.
\end{proof}

Note that $\U_0(L)$ contains $\widetilde{\U}(L)$ because $g_i^2$ are modular with trivial character for $\widetilde{\U}(L)$.  Let $\chi$ be the character of $f=f_1f_2f_3f_4$ on $\U_0(L)$. Since $f$ has a nontrivial character under all tetraflections in $\U_f$, we have  $\ker(\chi) \le \widetilde{\U}(L)$. On the other hand, $F\in M_4(\widetilde{\Orth}^+(L_\ZZ))$, and therefore $f\in M_4(\widetilde{\U}(L))$, so $\widetilde{\U}(L) = \ker(\chi)$. Using the decomposition of the Jacobian (Theorem \ref{th:freeJacobian} (4)) we obtain
\begin{theorem}
$$
M_*(\widetilde{\U}(L))=\CC[g_1^2, g_2^2, g_3^2, f] \; \text{is generated in weights} \; 2,2,2,4 \; \text{with a relation in weight} \; 8.
$$
\end{theorem}

We denote the multiplier systems of $f_i^{1/2}$ on $\U_0(L)$ by $\nu_i$. We define
$$
\U_1(L)=\{ \gamma \in \U_0(L) : \nu_2(\gamma)=\nu_3(\gamma)=\nu_4(\gamma) \}.
$$
We denote the common restriction of $\nu_2$, $\nu_3$ and $\nu_4$ to $\U_1(L)$ by $\nu$, such that $f_i^{1/2} \in M_{1/2}(\U_1(L),\nu)$ for $i=2,3,4$. The reflections contained in $\U_1(L)$ are exactly the tetraflections associated with $\div(f_1)$ and the biflections associated with $\div(g_1g_2g_3)$, and $\nu$ is trivial on all of them. 
We find that $J_{\U}(f_2^{1/2}, f_3^{1/2}, f_4^{1/2})=f_1^{3/2}g_1g_2g_3$ up to scalar. There is a natural generalization of Theorem \ref{th:modularJacobian} to modular forms of rational weights and multiplier system (analogous to \cite[Theorem 2.5]{WW20b} for orthogonal groups). This leads to the following:
\begin{theorem}
$$
M_*(\U_1(L), \nu):=\bigoplus_{k\in \ZZ} M_{k/2}(\U_1(L), \nu^k) = \CC[f_2^{1/2}, f_3^{1/2}, f_4^{1/2}] \; \text{is generated in weights} \; 1/2, 1/2, 1/2.
$$
\end{theorem}
We further define 
$$
\U_2(L)=\{ \gamma \in \U_0(L) : \nu_1(\gamma)= \nu_2(\gamma)=\nu_3(\gamma)=\nu_4(\gamma) \}.
$$
From the decomposition of the Jacobian for $\U_1(L)$, we conclude
\begin{theorem}
$$
M_*(\U_2(L), \nu) = \CC[f_1^{1/2}, f_2^{1/2}, f_3^{1/2}, f_4^{1/2}]
$$
modulo a unique relation among the generators:
$$
f_1^2+f_2^2+f_3^2+f_4^2=0.
$$
In particular the Baily-Borel compactification of the ball quotient $\cD_{\U} / \U_2(L)$ is the Fermat quartic surface $x^4+y^4+z^4+w^4=0$ in $\PP^3(\CC)$. \end{theorem}

\section{A free algebra of modular forms on a ball quotient over \texorpdfstring{$\ZZ[\sqrt{-2}]$}{}}\label{sec:sqrt2}

We are aware of only one example of an arithmetic group $\Gamma \le \U(n, 1)$, defined over a base field $\F$ of discriminant $D_{\F} \ne -3, -4$, for which the associated algebra of modular forms is freely generated. This is the unitary group of the root lattice $L_{\ZZ} = U \oplus U(2) \oplus D_4$ viewed as a Hermitian lattice over $\mathbb{Z}[\sqrt{-2}]$. In this section we will show that the ring $M_*(\U(L))$ is the polynomial algebra on generators of weights $2, 8, 10, 16$.

The Hermitian lattice $L$ can be realized as $\mathbb{Z}[\sqrt{-2}]^4$ with Gram matrix $$\begin{psmallmatrix} 0 & 0 & 0 & 1/2 \\ 0 & 1 & (1 + \sqrt{-2})/2 & 0 \\ 0 & (1 - \sqrt{-2}) / 2 & 1 & 0 \\ 1/2 & 0 & 0 & 0 \end{psmallmatrix}.$$ There are 10 Borcherds products $F_1,...,F_{10}$ of weight $4$ associated to this lattice, each with input form with principal part $$8 e_0 + q^{-1/2} e_v,$$ where $v$ runs through the $10$ cosets of $L'/L$ of norm $1/2 + \mathbb{Z}$, and there is a weight 40 product $F$ with principal part $80 e_0 + q^{-1} e_0$. If $f_i$ and $f$ denote their restrictions to the unitary group then we have $$f = f_1 \cdot ... \cdot f_{10},$$ and $$r^{\perp} \in \mathrm{div}(f) \; \text{if and only if} \; \sigma_{r, -1} \in \U(L).$$

Recall from \cite[Theorem 3.20]{WW20c} that the ring of modular forms for $\Orth^+(L_{\ZZ})$ is the polynomial algebra on generators $$\cE_2, \cE_6, p_2, \cE_{10}, p_3, p_4, p_5,$$ where $\cE_6, \cE_{10}$ are the standard Eisenstein series, $\cE_2$ is the theta lift of the (unique) normalized Weil invariant, and $$p_k = e_1^k + e_2^k + e_3^k + e_4^k + e_5^k$$ where $e_1,...,e_5$ are the five holomorphic Eisenstein series of weight four,  normalized such that $\Orth^+(L)$ acts on them by permutations; and that these forms satisfy $$\cE_2 = \sqrt{e_1 + e_2 + e_3 + e_4 + e_5}.$$ By abuse of notation we also write $\cE_k$, $p_k$ for their unitary restrictions.

\begin{theorem} $$M_*(\U(L)) = \mathbb{C}[\cE_2, p_2, \cE_{10}, p_4].$$ The Jacobian of the generators is a nonzero constant multiple of $f$.
\end{theorem}
\begin{proof} Clearly $J_{\U} = J_{\U}(\cE_2, p_2, \cE_{10}, p_4)$ has weight $40$ and vanishes on all mirrors $r^{\perp}$ with $\sigma_{r, -1} \in \U(L)$. The proof follows immediately from Theorem \ref{th:modularJacobian} as soon as we can show that $J_{\U}$ is not identically zero. To this end we consider the Taylor expansions of these forms on the standard Siegel domain: these are power series in two variables $z_1, z_2$ with coefficients in the ring $$M_*(\Gamma_0(2)) = \mathbb{C}[e_2, e_4],$$ where $e_2(\tau) = 2E_2(2\tau) - E_2(\tau) = 1 + 24q + ...$ and $e_4(\tau) = E_4(2\tau) = 1 + 240q^2 + ...$ One can compute that the quasi-pullback of the product $f$ to $U \oplus U(2)$ has weight $52$, so its Taylor coefficients vanish in degree less than $12$. Since the computation is not very difficult, we worked out the power series expansions of all generators of $M_*(\Orth^+(L_{\ZZ}))$ to degree $19$. These are presented in an auxiliary file. Using these power series one can check that the Jacobian is indeed nonzero.
\end{proof}

\begin{remark} For Hermitian lattices $L$ over $\mathbb{Z}[i]$ or $\mathbb{Z}[e^{2\pi i / 3}]$ such that $M_*(\U(L))$ and $M_*(\Orth^+(L_{\ZZ}))$ are both free, the unitary domain $\cD_{\U}$ was cut out of $\cD$ as the simultaneous zero locus of modular forms of certain weights. Here the computation of the vanishing ideal of $\cD_{\U}$ seems nontrivial. Using the Taylor expansions from the proof of the above theorem, one can show that $\cD_{\U}$ is exactly the simultaneous zero locus of the orthogonal modular forms $$16 \cE_6 - 7 \cE_2^3,$$ $$3200 p_3 - 1632 \cE_2 \cE_{10} - 300 \cE_2^2 p_2 + 115 \cE_2^6$$ and \begin{align*}&153600 p_5 - 163200 \cE_2 p_2 \cE_{10} - 1680000 \cE_2^2 p_4 + 315000 \cE_2^2 p_2^2 \\&+ 437920 \cE_2^5 \cE_{10} - 193500 \cE_2^6 p_2 - 6575 \cE_2^{10} - 147968 \cE_{10}^2.\end{align*}
\end{remark}

\section{Some interesting non-free algebras of unitary modular forms}\label{sec:nonfree}
Theorem \ref{th:freeJacobian} (4) determines the irreducible factorization of the Jacobian of the generators of a free algebra of modular forms for $\Gamma$. This factorization determines the character group of $\Gamma$ and can be used to compute interesting algebras of modular forms for certain subgroups of $\Gamma$. We give several examples below.

\subsection{The lattice \texorpdfstring{$U\oplus U(3)\oplus A_2$}{} over \texorpdfstring{$\ZZ[\omega]$}{}.} There are six Borcherds products of weight 3 on $U\oplus U(3)\oplus A_2$ with principal parts of the form
$$
6e_0+q^{-1/3}(e_v+e_{-v}), \quad (v,v)=2/3,\; \ord(v)=3.
$$
We label their unitary restrictions $f_i$, $1\leq i\leq 6$. The reflective Borcherds product with principal part $36e_0+q^{-1}e_0$ then restricts to $\prod_{i=1}^6 f_i$. Recall from \S\ref{sec:twins} that $M_*(\widetilde{\U}(L))$ is freely generated by 3 forms of weight 3 whose Jacobian is $$J_{\U}=\prod_{i=1}^6 f_i^{2/3},$$ that the forms $f_i$ span a 3-dimensional space, and that any linearly independent three $f_i$ can be taken as a set of generators. In view of the irreducible decomposition of the Jacobian, the character group has order $3^6 = 729$, each character determined by uniquely its values on the triflections associated to $\mathrm{div}(f_i)$. If $\widetilde{\U}'(L)$ is the commutator subgroup of $\widetilde{\U}(L)$, then 
$$
M_*(\widetilde{\U}'(L)) = \CC[f_i^{1/3}, 1\leq i \leq 6] / R
$$
with an ideal of relations $R$. Since the generators have distinct characters on $\widetilde{\U}(L)$, any relation among them must be of the form $$P(f_1,...,f_6) = 0.$$ (That is,  $f_i^{1/3}$ only appears with exponent a multiple of three). By applying \cite[Theorem 5.6]{WW20c} we see that all relations among $X_i = f_i^{1/3}$ are generated from three three-term cubic relations:
$$
X_1^3+X_6^3=X_5^3, \quad X_4^3+X_6^3=X_2^3, \quad X_4^3+X_5^3=X_3^3.
$$

\subsection{The lattice \texorpdfstring{$2U(3)\oplus A_2$}{} over \texorpdfstring{$\ZZ[\omega]$}{}.} Recall from \S\ref{sec:twins} that $M_*(\widetilde{\U}(L))$ for $L_{\ZZ} = 2U(3) \oplus A_2$ is freely generated by 3 forms of weight 1. The lattice $L$ can be realized as $\mathbb{Z}[e^{2\pi i / 3}]$ with the diagonal Gram matrix $\mathrm{diag}(1, 1, -1)$, and the discriminant kernel $\widetilde{\U}(L)$ is exactly the principal congruence subgroup of level $1 - e^{2\pi i / 3}$. In this context the structure is an earlier-known theorem of Feustel and Holzapfel, who also determined the structure of several related rings of modular forms; for details see the book \cite{H86} and Shiga's paper \cite{Shi88}.

There are 45 Borcherds products of singular weight $1$ on $2U(3)\oplus A_2$ with principal parts of the form
$$
2e_0+q^{-1/3}(e_v+e_{-v}), \quad (v,v)=2/3,\; \ord(v)=3.
$$
The unitary restrictions of nine of these products, denoted $f_1,...,f_9$, have zeros of multiplicity 3 and therefore have well-defined cube roots. The remaining 36 products restrict to 12 distinct unitary Borcherds products $g_1,...,g_{12}$, each with zeros only of multiplicity 1.

The restriction of the Borcherds product with principal part $18e_0+q^{-1}e_0$ is equal to $\prod_{i=1}^9 f_i$ up to scalar. These products $f_i$ span a 3-dimensional space and any three that are linearly independent form a set of generators. Let $\nu_i$ be the multiplier system of $f_i^{1/3}$ on $\widetilde{\U}(L)$, and let $\widetilde{\U}'(L)$ be the subgroup of $\widetilde{\U}(L)$ on which all $\nu_i$ coincide. We denote the common restriction of $\nu_i$ to $\widetilde{\U}'(L)$ by $\nu$. Recall that the Jacobian associated with $M_*(\widetilde{\U}(L))$ is $\prod_{i=1}^9f_i^{2/3}$.  From the irreducible decomposition of the Jacobian, we conclude
$$
M_*(\widetilde{\U}'(L), \nu):= \bigoplus_{k\in \ZZ} M_{k/3}(\widetilde{\U}'(L), \nu^k)  = \CC[f_i^{1/3}, 1\leq i \leq 9] / R,
$$
and all relations among the nine generators $X_i = f_i^{1/3}$ are generated by six three-term cubic relations which are defined over $\mathbb{Z}[e^{2\pi i / 3}]$. With respect to the ordering we used we found the relations
\begin{align*}
&X_1^3 + X_6^3 = X_9^3 & 
&X_7^3 + X_8^3 = X_9^3 &
&X_1^3 + X_2^3 = X_3^3 & \\
&X_3^3 + X_5^3 = X_7^3 & 
&X_6^3 + \zeta_3 X_8^3 = X_3^3 &
&X_3^3 + X_4^3 = \zeta_3X_1^3 &     
\end{align*}
where $\zeta_3=e^{2\pi i /3}$.

\subsection{The lattice \texorpdfstring{$2U(2)\oplus D_4$}{} over \texorpdfstring{$\ZZ[i]$}.}\label{sec:2U2+D4extra}  There are 36 Borcherds products of singular weight $2$ on $2U(2)\oplus D_4$ with principal parts of the form
$$
4e_0+q^{-1/2}e_v, \quad (v,v)=1,\; \ord(v)=2.
$$
The unitary restrictions of $12$ of these products, labelled $f_1,...,f_{12}$, have zeros only of multiplicity two and therefore holomorphic square roots. The remaining $24$ products restrict to $12$ distinct unitary products, labelled $g_1,...,g_{12}$, all with zeros only of multiplicity one.

The restriction of the Borcherds product with principal part $48e_0+q^{-1}e_0$ equals $\prod_{i=1}^{12} f_i$. Recall from \S\ref{sec:twins} that $M_*(\widetilde{\U}(L))$ is freely generated by 4 forms of weight 2 whose Jacobian is $J_{\U}=\prod_{i=1}^{12} f_i^{1/2}$. In particular the products $f_i$ span a four-dimensional space and any four that span this space can be taken as algebra generators.  In view of the irreducible decomposition of the Jacobian, the character group of $\widetilde{\U}(L)$ has order $2^{12}$. Let $\widetilde{\U}'(L)$ be the commutator subgroup of $\widetilde{\U}(L)$.  Then
$$
M_*(\widetilde{\U}'(L)) = \CC[f_i^{1/2}, 1\leq i \leq 12] / R,
$$
where $R$ is generated by $8$ three-term quadratic relations, i.e. linear relations among the $f_i$. With respect to our ordering of $X_i = f_i^{1/2}$ we found the relations
\begin{align*}
&X_1^2+X_{10}^2=X_{12}^2& &X_2^2+X_{11}^2=X_{10}^2& &X_3^2+X_{8}^2=X_{12}^2& &X_4^2+X_{11}^2=X_{8}^2& \\
&X_5^2+X_{10}^2=X_{8}^2& &X_6^2+X_{7}^2=X_{10}^2& &X_1^2+X_{11}^2=X_{9}^2& &X_4^2+X_{7}^2=X_{12}^2&
\end{align*}

\subsection{The lattice \texorpdfstring{$U\oplus U(2)\oplus D_4$}{} over \texorpdfstring{$\ZZ[i]$}.} There are 10 Borcherds products of weight 4 on $U\oplus U(2)\oplus D_4$ with principal parts of the form
$$
8e_0+q^{-1/2}e_v, \quad (v,v)=1,\; \ord(v)=2.
$$
We denote their unitary restrictions by $f_i$ for $1\leq i\leq 10$. 
The unitary restriction of the Borcherds product with principal part $80e_0+q^{-1}e_0$ equals $\prod_{i=1}^{10}f_i$. Recall from \S\ref{sec:twins} that $M_*(\widetilde{\U}(L))$ is freely generated by 4 forms of weight 4 and their Jacobian is $J_{\U}=\prod_{i=1}^{10} f_i^{1/2}$. In particular these $f_i$ span a 4-dimensional space and any four that span the space can be taken as algebra generators.  In view of the irreducible decomposition of the Jacobian, there are exactly $2^{10}$ characters of $\widetilde{\U}(L)$. Let $\widetilde{\U}'(L)$ be the commutator subgroup of $\widetilde{\U}(L)$.  Then
$$
M_*(\widetilde{\U}(L)) = \CC[f_i^{1/2}, 1\leq i \leq 10] / R
$$
Similarly to the case $U\oplus U(3)\oplus A_2$, in the notation in \cite[Remark 3.19]{WW20c} all relations among the ten generators $X_i = f_i^{1/2}$ are determined by the six three-term quadratic relations
\begin{align*}
&X_4^2+X_6^2=X_1^2,& &X_2^2+X_5^2=X_4^2,& &X_3^2+X_5^2=X_1^2,& \\
&X_4^2+X_7^2=X_9^2,& &X_5^2+X_8^2=X_9^2,& &X_7^2+X_{10}^2=X_6^2.&
\end{align*}

\section{Open questions and conjectures}
There are several questions related to our work in this paper that we have been unable to answer.
\begin{enumerate}
    \item In \cite{VS17}, Vinberg and Shvartsman proved that the algebra of modular forms for $\Orth(n,2)$ is never free when $n>10$. It would be interesting to know if there is an analogue of this in the unitary case. The Jacobian of the generators of a free algebra of unitary modular forms vanishes precisely on mirrors of reflections. It is natural to guess that any such modular form arises from a reflective (orthogonal) Borcherds product by restriction. The classification of reflective modular forms on type IV symmetric domains leads to the following more precise conjecture:
    \begin{conjecture}
    \begin{itemize}
        \item[(i)] When $d=-1$ or $-3$, the algebra of modular forms for a finite-index subgroup of some $\U(L)$ is never free if $n>5$. 
        \item[(ii)] When $\abs{d}>3$, the algebra of modular forms for a finite-index subgroup of some $\U(L)$ is never free. 
    \end{itemize}
    \end{conjecture}
    \item Similarly to \cite[Conjecture 5.2]{Wan20}, we have the following conjecture.
    \begin{conjecture}
    Let $\Gamma<\U(L)$ be a finite-index subgroup and $\Gamma_1$ be a reflection subgroup of $\U(L)$ containing $\Gamma$. If $M_*(\Gamma)$ is free, then the smaller algebra $M_*(\Gamma_1)$ is also free.
    \end{conjecture}
    \item In Theorem \ref{th:twins}, one of our assumptions is that there are $n$ generators whose restrictions are identically zero. Can we remove this assumption in the theorem? Equivalently, if $M_*(\Gamma)$ is free, is $M_*(\Gamma_{\U})$ also free? Here, $\Gamma= \widetilde{\Orth}^+(L_\ZZ)$ or $\Orth^+(L_\ZZ)$.
\end{enumerate}

\bigskip

\noindent
\textbf{Acknowledgements} 
The authors would like to thank Daniel Allcock, Fabien Cl\'ery, Eberhard Freitag and Riccardo Salvati Manni for helpful discussions. H. Wang is grateful to Max Planck Institute for Mathematics in Bonn for its hospitality and financial support. 

\section{Appendix: Tables}

On the following pages we list the algebras of modular forms for Eisenstein and Gaussian lattices determined in this paper. In each table, $L_{\mathbb{Z}}$ is the underlying $\mathbb{Z}$-lattice. $\mathbf{S}$ is a possible Hermitian Gram matrix for $L$. The modular groups $\widetilde{\U}$, $\U$, $\U_r$ denote the discriminant kernel, unitary group and maximal reflection subgroup, respectively; for the other modular groups the reader is referred to the section in the final column. The fourth and fifth columns contain the weights of the generators and the weights of the defining relations (if any).

\clearpage
\begin{table}[htbp]
\caption{Algebras of modular forms on Eisenstein lattices.}
\label{tab:appendix1}
\renewcommand\arraystretch{1.1}
\noindent\[
\begin{array}{cccccc}
L_{\mathbb{Z}} & \mathbf{S} & \Gamma & \text{Generators} & \text{Relations} &  \text{Reference}\\
\hline
& \\[-4mm]
\multirow{2}{*}{$2U\oplus A_2$} & \multirow{2}{*}{$\begin{psmallmatrix} 0 & 0 & i/ \sqrt{3} \\ 0 & 1 & 0 \\ -i/ \sqrt{3} & 0 & 0 \end{psmallmatrix} $ }  & \widetilde{\U} & 6, 9, 12 & - & \text{Table \ref{tab:3kernel}} \\
		  &	 & \U & 6, 12, 18 & - & \text{Table \ref{tab:3full}} \\
\hline
& \\[-4mm]
\multirow{2}{*}{$U \oplus U(3) \oplus A_2$} & \multirow{2}{*}{$ \begin{psmallmatrix} 0 & 0 & 1 \\ 0 & 1 & 0 \\ 1 & 0 & 0 \end{psmallmatrix}$} & \widetilde{\U} & 3, 3, 3 & - & \text{Table \ref{tab:3kernel}} \\
 & & \U & 6, 12, 18 & - & \text{Theorem \ref{th:3extra}} \\
\hline
& \\[-4mm]
\multirow{2}{*}{$2U(2) \oplus A_2$} & \multirow{2}{*}{$ \begin{psmallmatrix} 0 & 0 & 2i/\sqrt{3} \\ 0 & 1 & 0 \\ -2i/\sqrt{3} & 0 & 0 \end{psmallmatrix}$} & \widetilde{\U} & 2, 2, 3 & - & \text{Table \ref{tab:3kernel}} \\
 & & \U & 6, 6, 12 & - & \text{Theorem \ref{th:3extra}} \\
\hline
& \\[-4mm]
\multirow{4}{*}{$ 2U \oplus A_2(2)$} & \multirow{4}{*}{$\begin{psmallmatrix} 0 & 0 & i/\sqrt{3} \\ 0 & 2 & 0 \\ -i/\sqrt{3} & 0 & 0 \end{psmallmatrix}$} & \U_1 & 1, 4, 6 & - & \S\ref{sec:2U+A22} \\
 & & \U_2 & 3, 6, 12 & - & \\
 & & \U_3 & 6, 6, 6 & - & \\
 & & \U & 6, 6, 12 & - & \\
\hline
& \\[-4mm]
2U(3) \oplus A_2 & \begin{psmallmatrix} 0 & 0 & i \sqrt{3} \\ 0 & 1 & 0 \\ -i\sqrt{3} & 0 & 0 \end{psmallmatrix} & \widetilde{\U} & 1, 1,1 & - & \text{Table \ref{tab:3kernel}} \\
\hline
& \\[-4mm]
\multirow{3}{*}{$2U \oplus A_2(3)$} & \multirow{3}{*}{$ \begin{psmallmatrix} 0 & 0 & i/\sqrt{3} \\ 0 & 3 & 0 \\ -i/\sqrt{3} & 0 & 0 \end{psmallmatrix} $} & \U_1 & 1, 4, 6 & - & \S\ref{sec:2U+A23} \\
 & & \U_r & 2, 4, 6 & - & \\
 & & \U & 6, 6, 6, 12 & 18 & \\
 \hline
 & \\[-4mm]
\multirow{2}{*}{$ 2U \oplus D_4$} & \multirow{2}{*}{$\begin{psmallmatrix} 0 & 0 & 0 & i/\sqrt{3} \\ 0 & 1 & i/\sqrt{3} & 0 \\ 0 & -i/\sqrt{3} & 1 & 0 \\ -i/\sqrt{-3} & 0 & 0 & 0 \end{psmallmatrix} $} & \widetilde{\U} & 6, 8, 12, 18 & - & \text{Table \ref{tab:3kernel}} \\
  & & \U & 6, 12, 18, 24 & - & \text{Table \ref{tab:3full}} \\ [3ex]
  \hline
  & \\[-4mm]
 \multirow{3}{*}{$2U \oplus 2A_2$} & \multirow{3}{*}{$\begin{psmallmatrix} 0 & 0 & 0 & i/\sqrt{3} \\ 0 & 1 & 0 & 0 \\ 0 & 0 & 1 & 0 \\ -i/\sqrt{3} & 0 & 0 & 0 \end{psmallmatrix}$} & \widetilde{\U} & 6,6,9,9,12 & 18 & \S\ref{sec:2U+2A2} \\
  & & \U_1 & 6,6,9,12 & - & \\
  & & \U & 6,12,12,18 & - & \\
  \hline
  & \\[-4mm]
 U \oplus U(3) \oplus 2A_2 & \begin{psmallmatrix} 0 & 0 & 0 & 1 \\ 0 & 1 & 0 & 0 \\ 0 & 0 & 1 & 0 \\ 1 & 0 & 0 & 0 \end{psmallmatrix} & \U & 6, 12, 18, 30 & - & \S\ref{sec:U+U3+2A2} \\ 
  \hline
  & \\[-4mm]
2U(2) \oplus D_4 & \begin{psmallmatrix} 0 & 0 & 0 & 2i/\sqrt{3} \\ 0 & 1 & i/\sqrt{3} & 0 \\ 0 & -i/\sqrt{3} & 1 & 0 \\ -2i/\sqrt{-3} & 0 & 0 & 0 \end{psmallmatrix} & \widetilde{\U} & 2, 2, 2, 6 & - & \text{Table \ref{tab:3kernel}} \\ 
  \hline
 & \\[-4mm]
\multirow{2}{*}{$ 2U \oplus D_4(2)$} & \multirow{2}{*}{$\begin{psmallmatrix} 0 & 0 & 0 & i/\sqrt{3} \\ 0 & 2 & 2i/\sqrt{3} & 0 \\ 0 & -2i/\sqrt{3} & 2 & 0 \\ -i/\sqrt{-3} & 0 & 0 & 0 \end{psmallmatrix} $} & \U_r & 6, 6, 6, 6 & - & \text{\S \ref{sec:2U+D423}} \\
  & & \U & 6, 6, 12, 12, 12 & 24 &  \\ [4ex]
  \hline
 & \\[-4mm]
 U \oplus U(3) \oplus 3A_2 & \begin{psmallmatrix} 0 & 0 & 0 & 0 & 1 \\ 0 & 1 & 0 & 0 & 0 \\ 0 & 0 & 1 & 0 & 0 \\ 0 &  0 & 0 & 1 & 0  \\ 1 & 0 & 0 & 0 & 0 \end{psmallmatrix} & \U & 6, 12, 18, 24, 30 & - & \S\ref{sec:U+U3+3A2} \\ 
  \hline  
 & \\[-4mm]
\multirow{2}{*}{$2U \oplus E_6$} & \multirow{2}{*}{$\begin{psmallmatrix} 0 & 0 & 0 & 0 & i/\sqrt{3} \\ 0 & 1 & i/\sqrt{3} & 0 & 0 \\ 0 & -i/\sqrt{3} & 1 & i/\sqrt{3} & 0 \\ 0 & 0 & -i/\sqrt{3} & 1 & 0 \\ -i/\sqrt{3} & 0 & 0 & 0 & 0 \end{psmallmatrix}$} & \widetilde{\U} & 6, 12, 15, 18, 24 & - & \text{Table \ref{tab:3kernel}} \\
  & & \U & 6, 12, 18, 24, 30 & - & \text{Theorem \ref{th:3extra}} \\ [6ex]
  \hline
\end{array} 
\]
\end{table}

\clearpage
\begin{table}[htbp]
\caption{Algebras of modular forms on Eisenstein lattices, continued.}
\label{tab:appendix2}
\renewcommand\arraystretch{1.1}
\noindent\[
\begin{array}{cccccc}
L_{\mathbb{Z}} & \mathbf{S} & \Gamma & \text{Generators} & \text{Relations} &  \text{Reference}\\
\hline
& \\[-4mm]  
\multirow{2}{*}{$2U \oplus 3A_2$} & \multirow{2}{*}{$\begin{psmallmatrix} 0 & 0 & 0 & 0 & i/\sqrt{3} \\ 0 & 1 & 0 & 0 & 0 \\ 0 & 0 & 1 & 0 & 0 \\ 0 & 0 & 0 & 1 & 0 \\ -i/\sqrt{3} & 0 & 0 & 0 & 0 \end{psmallmatrix}$} & \U_1 & 3, 6, 6, 9, 12 & - & \S\ref{sec:2U+3A2} \\
  & & \U & 6, 6, 12, 12, 18 & - & \\ [3ex]
  \hline
  & \\[-4mm]
  2U \oplus E_8 & \begin{psmallmatrix} 0 & 0 & 0 & 0 & 0 & i/\sqrt{3} \\ 0 & 1 & i/\sqrt{3} & 0 & 0 & 0  \\ 0 & -i/\sqrt{3} & 1 & i/\sqrt{3} & 0 & 0 \\ 0 & 0 & -i/\sqrt{3} & 1 & i/\sqrt{3} & 0 \\ 0 & 0 & 0 & -i/\sqrt{3} & 1 & 0 \\ -i/\sqrt{3} & 0 & 0 & 0 & 0 & 0 \end{psmallmatrix} & \U & 12, 18, 24, 30, 36, 42 & - & \text{Table \ref{tab:3full}} \\ 
  \hline 
\end{array} 
\]
\end{table}

\begin{table}[htbp]
\caption{Algebras of modular forms on Gaussian lattices.}
\label{tab:appendix3}
\renewcommand\arraystretch{1.1}
\noindent\[
\begin{array}{cccccc}
L_{\mathbb{Z}} & \mathbf{S} & \Gamma & \text{Generators} & \text{Relations} &  \text{Reference}\\
\hline
& \\[-4mm]
\multirow{4}{*}{$2U \oplus 2A_1$} & \multirow{4}{*}{$\begin{psmallmatrix} 0 & 0 & 1/2 \\ 0 & 1 & 0 \\1/2 & 0 & 0\end{psmallmatrix}$} & \widetilde{\U}_1 & 2, 3, 4 & - & \S\ref{sec:2U+2A1} \\
 & & \widetilde{\U}_r & 4, 4, 6 & - & \\
 & & \widetilde{\U} & 4, 8, 10, 12 & 20 & \\
 & & \U & 4, 8, 12 & - & \\
  \hline
& \\[-4mm] 
\multirow{3}{*}{$2U \oplus 2A_1(2)$} & \multirow{3}{*}{$\begin{psmallmatrix} 0 & 0 & 1/2 \\ 0 & 2 & 0 \\ 1/2 & 0 & 0 \end{psmallmatrix}$} & \U_1 & 1, 3, 4 & - & \S\ref{sec:2U+2A12} \\
  & & \U_r & 2, 4, 6 & - & \\
  & & \U & 4, 4, 8, 12 & 16 & \\
  \hline
& \\[-4mm]  
\multirow{7}{*}{$ U \oplus U(2) \oplus 2A_1$} & \multirow{7}{*}{$ \begin{psmallmatrix} 0 & 0 & (1+i)/2 \\ 0 & 1 & 0 \\  (1-i)/2 & 0 & 0 \end{psmallmatrix} $} & \widetilde{\U}_1 & 1, 1, 1 & - & \S\ref{sec:U+U2+2A1} \\
  & & \widetilde{\U}_r & 2, 2, 2 & - & \\
  & & \widetilde{\U} & 4, 4, 4, 6 & 12 & \\
  & & \U_1 & 2, 2, 2 & - & \\
  & & \U_2 & 3, 4, 8 & - & \\
  & & \U_3 & 4, 4, 4 & - & \\
  & & \U & 4, 8, 12 & - & \\
  \hline
& \\[-4mm]  
\multirow{6}{*}{$2U(2) \oplus 2A_1$} & \multirow{6}{*}{$\begin{psmallmatrix} 0 & 0 & 1 \\ 0 & 1 & 0 \\ 1 & 0 & 0 \end{psmallmatrix}$} & \U_1 & 1/2, 1/2, 1/2 & - & \S\ref{sec:2U2+2A1} \\
  & & \U_2 & 1/2, 1/2, 1/2, 1/2 & 2 & \\
  & & \U_f & 1, 1, 1 & - & \\
  & & \U_0 & 2, 2, 2 & - & \\
  & & \widetilde{\U} & 2, 2,2, 4 & 8 & \\
  & & \Gamma & 4, 4, 4 & - & \\
  \hline
& \\[-4mm]    
\multirow{2}{*}{$ 2U \oplus D_4$} & \multirow{2}{*}{$\begin{psmallmatrix} 0 & 0 & 0 & 1/2 \\ 0 & 1 & (1+i)/2 & 0 \\ 0 & (1-i)/2 & 1 & 0 \\ 1/2 & 0 & 0 & 0 \end{psmallmatrix}$} & \widetilde{\U} & 4, 8, 8, 12 & - & \text{Table \ref{tab:1kernel}} \\
 & & \U & 4, 12, 16, 24 & - & \text{Table \ref{tab:1full}} \\ [3ex]
 \hline
   & \\[-4mm] 
2U(2) \oplus D_4 & \begin{psmallmatrix} 0 & 0 & 0 & 1 \\ 0 & 1 & (1+i)/2 & 0 \\ 0 & (1-i)/2 & 1 & 0 \\ 1 & 0 & 0 & 0 \end{psmallmatrix} & \widetilde{\U} & 2, 2, 2, 2 &  - & \text{Table \ref{tab:1kernel}} \\
  \hline
\end{array} 
\]
\end{table}

\clearpage

\begin{table}[htbp]
\caption{Algebras of modular forms on Gaussian lattices, continued.}
\label{tab:appendix4}
\renewcommand\arraystretch{1.1}
\noindent\[
\begin{array}{cccccc}
L_{\mathbb{Z}} & \mathbf{S} & \Gamma & \text{Generators} & \text{Relations} &  \text{Reference}\\
 \hline
& \\[-4mm]  
 \multirow{2}{*}{$ U \oplus U(2) \oplus D_4$} & \multirow{2}{*}{$ \begin{psmallmatrix} 0 & 0 & 0 & (1+i)/2 \\ 0 & 1 & (1+i)/2 & 0 \\ 0 & (1-i)/2 & 1 & 0 \\ (1-i)/2 & 0 & 0 & 0 \end{psmallmatrix} $} & \widetilde{\U} & 4, 4, 4, 4 & - & \text{Table \ref{tab:1kernel}} \\
  & & \U & 8, 12, 16, 20 & - & \text{Table \ref{tab:1full}} \\ [3ex]
  \hline
  & \\[-4mm] 
  2U \oplus 4A_1 & \begin{psmallmatrix} 0 & 0 & 0 & 1/2 \\ 0 & 1 & 0 & 0 \\ 0 & 0 & 1 & 0 \\ 1/2 & 0 & 0 & 0 \end{psmallmatrix} & \U & 4, 4, 8, 12 & - & \text{Table \ref{tab:1full}} \\
  \hline
  & \\[-4mm] 
 \multirow{2}{*}{ $2U \oplus D_4(2)$} & \multirow{2}{*}{$\begin{psmallmatrix} 0 & 0 & 0 & 1/2 \\ 0 & 2 & 1+i & 0 \\ 0 & 1-i & 2 & 0 \\ 1/2 & 0 & 0 & 0 \end{psmallmatrix}$} & \U_r & 4, 4, 6, 6 & - & \S\ref{sec:2U+D42} \\
  & & \U & 4, 4, 12, 12, 12 & 24 &  \\ [2ex]
  \hline
  & \\[-4mm] 
\multirow{2}{*}{$2U \oplus D_6$} & \multirow{2}{*}{$\begin{psmallmatrix} 0 & 0 & 0 & 0 & 1/2 \\ 0 & 2 & (1+i)/2 & 1/2 & 0 \\ 0 & (1-i)/2 & 1 & (1+i)/2 & 0 \\ 0 & 1/2 & (1-i)/2 & 1 & 0 \\ 1/2 & 0 & 0 & 0 & 0 \end{psmallmatrix}$} & \widetilde{\U} & 4, 6, 8, 12, 16 & - & \text{Table \ref{tab:1kernel}} \\
  & & \U & 4, 8, 12, 12, 16 & - & \text{Table \ref{tab:1full}} \\ [4ex]
  \hline
  & \\[-4mm] 
  2U \oplus E_8 & \begin{psmallmatrix} 0 & 0 & 0 & 0 & 0 & 1/2 \\ 0 & 1 & 1/2 & 1/2 & 0  & 0 \\ 0 & 1/2 & 1 & 1/2 & 1/2 & 0 \\ 0 & 1/2 & 1/2 & 1 & (1+i)/2 & 0 \\ 0 & 0 & 1/2 & (1-i)/2 & 1 & 0 \\ 1/2 & 0 & 0 & 0 & 0 &0 \end{psmallmatrix} & \U & 4, 12, 16, 24, 28, 36 & - & \text{Table \ref{tab:1full}} \\
  \hline
  & \\[-4mm] 
 \multirow{2}{*}{$ 2U \oplus D_8$} & \multirow{2}{*}{$\begin{psmallmatrix} 0 & 0 & 0 & 0 & 0 & 1/2 \\ 0 &1 & 1/2 & 1/2 & 1/2 & 0 \\ 0 & 1/2 & 1 & 1/2 & 1/2 & 0 \\ 0 & 1/2 & 1/2 & 1 & (1+i)/2 & 0 \\ 0 & 1/2 & 1/2 & (1-i)/2 & 1 & 0 \\ 1/2 & 0 & 0 & 0 & 0 & 0 \end{psmallmatrix}$} & \widetilde{\U} & 4, 4, 8, 12, 12, 16 & - & \text{Table \ref{tab:1kernel}} \\
  & & \U & 4, 8, 8, 12, 12, 16 & - & \text{Table \ref{tab:1full}} \\ [6ex]
  \hline
\end{array} 
\]
\end{table}

\bibliographystyle{plainnat}
\bibliofont
\bibliography{refs}

\end{document}